\documentclass[11pt]{amsart}
\usepackage[parfill]{parskip}    
\usepackage{fullpage}
\usepackage{graphicx}
\usepackage{amssymb,amsmath,latexsym}
\usepackage{tikz}
\usepackage{amsfonts,amssymb}
\usetikzlibrary{arrows,decorations.markings,decorations.pathreplacing,calc,matrix,intersections}
\tikzset{->-/.style={decoration={markings,mark=at position #1 with {\arrow{>}}},postaction={decorate}}}

\usepackage{epstopdf}
\usepackage{ifpdf}
\usepackage{color}
\usepackage{float}
\usepackage{amscd,amsmath, mathabx}
\usepackage{amssymb,amsmath,latexsym,color,enumerate,tikz}
\usepackage[all]{xypic}       
\usepackage[varg]{pxfonts}
\usepackage{amsmath,amsthm,amsfonts,amssymb,fancyhdr,graphics,relsize,tikz-cd,mathtools,faktor,tikz,changepage}
\usepackage{graphicx}
\usepackage{placeins}

\usepackage{dsfont}
\definecolor{red}{rgb}{1,0,0} 

 \definecolor{darkgreen}{rgb}{0, .7, 0}

 \definecolor{purple}{rgb}{.7, 0, 1}

\newcommand{\Z}{{\mathbb{Z}}}
\newcommand{\Q}{{\mathbb{Q}}}
\newcommand{\R}{{\mathbb{R}}}
\newcommand{\C}{{\mathbb{C}}}

\tikzset{mynode/.style={draw,circle,fill=black,inner sep=2pt,outer sep=0.5pt}}

\newtheorem{theorem}{Theorem}[section]
\newtheorem*{theorem*}{Theorem}
\newtheorem*{lemma*}{Lemma}
\newtheorem{proposition}[theorem]{Proposition}
\newtheorem{lemma}[theorem]{Lemma}
\newtheorem{corollary}[theorem]{Corollary}

\theoremstyle{definition}
\newtheorem{definition}[theorem]{Definition}

\newtheorem{examples}[theorem]{Examples}

\theoremstyle{remark}
\newtheorem{remark}[theorem]{Remark}

\usepackage{hyperref}
 
\begin{document}
\title{Higher dimensional algebraic fiberings of group extensions}
\author{Dessislava H. Kochloukova}
\address
{Department of Mathematics, State University of Campinas (UNICAMP), 13083-859, Campinas, SP, Brazil }
\email{desi@unicamp.br}

\author{Stefano Vidussi}
\address{Department of Mathematics, University of California, Riverside, CA 92521, USA\\} 
\email{ svidussi@ucr.edu}
\keywords{Algebraic fibering, Coherence, Property $FP_m$, Bieri-Neumann-Strebel-Renz invariants}

\maketitle

\begin{abstract} We prove some conditions for the existence of higher dimensional algebraic fibering of group extensions. This leads to various corollaries on incoherence of groups and some geometric examples of algebraic fibers of type $F_n$ but not $FP_{n+1}$ of some groups including  pure braid groups and families of poly-surface groups that are fundamental groups of complex projective varieties.
\end{abstract}

\section{Introduction}

In \cite[Theorem 1]{F-V} Friedl and the second author established a criterion for algebraic fibering of a group $G$ i.e. the existence of an epimorphism $G \to \mathbb{Z}$ with finitely generated kernel, that we refer to as {\em algebraic fiber}.  A special case of \cite[Theorem 1]{F-V} was earlier proved by Kropholler and Walsh in \cite{K-W}. The proof relies on the Bieri-Neumann-Strebel $\Sigma^1$-invariant. This invariant was first defined in dimension 1 but as there are higher dimensional homotopical and homological  versions we show higher dimensional versions of  \cite[Theorem 1]{F-V}.  There is an obstacle to the original argument in  \cite[Theorem 1]{F-V} since quotients of finitely presented groups or of type $FP_2$ are not necessary of the same type but we resolve the problem using a homotopical criterion due to Kuckuck \cite{K}  and its homological version due to the first author and Lima \cite{K-L}.

We recall that  a group $G$ is of type $F_m$ if there is a $K(G,1)$ complex with finite $m$-skeleton. Type $F_1$ is equivalent with the finite generation of $G$. Type $F_2$ is equivalent with the finite presentability of $G$.
 
 A group $G$ is of type $FP_m$ if the trivial $\mathbb{Z} G$-module $\mathbb{Z}$ has a projective resolution with all modules in dimension $\leq m$ being finitely generated. $FP_1$ is equivalent with $G$ being finitely generated, $FP_2$ is equivalent with the relation module (associated with a finite generating set of $G$) being finitely generated as $\mathbb{Z} G$-module via conjugation.
 A group of type $F_m$ is of type $FP_m$ but if $m \geq 2$ the converse does not hold, see \cite{B-B}.
 If $m \geq 2$ then a group $G$ is of type $F_m$ if and only if $G$ is $FP_m$ and finitely presented.

 We state below our main result. In the case $n_0 = 1$ it coincides with  \cite[Theorem 1]{F-V}. 
 
\begin{theorem} \label{Main1} Let $1 \to K \to G \to \Gamma \to 1$ be a short exact sequence of groups such that $G$ and $K$ are of type $F_{n_0}$ (resp. $FP_{n_0}$) and the abelianization $\Gamma^{ab} = \Gamma / \Gamma'$ is infinite. Suppose that $\phi \colon K \to \R$ is  a discrete character  with kernel $Ker(\phi) = N$ of type $F_{n_0-1}$ (resp. $FP_{n_0-1}$) such that  $\phi$  extends to a real homomorphism of $G$. Then there exists a discrete homomorphism $\psi \colon G \to \mathbb{R}$ that extends $\phi$ such that $M = Ker (\psi)$ is of type $F_{n_0}$ (resp. $FP_{n_0}$). Furthermore if $K$, $G$ and $N$ are of type $F_{\infty}$ (resp. $FP_{\infty}$) then $M$ can be chosen of type $F_{\infty}$ (resp. $FP_{\infty}$).
\end{theorem}

There is a lot of freedom in the choice of  the map $\psi$ from Theorem \ref{Main1}. In fact $\psi$ can be chosen as any extension of $\phi$ that is  a rational point of a specific cone of characters. We discuss this fact in more details in Remark \ref{rem:cone}.

\medskip
We say that a group $G$ is $n$-coherent if any subgroup of $G$ that is of  type $F_n$ should be  of type $F_{n+1}$.  Thus a group is 1-coherent if it is coherent. We say that a group $G$ is homologically $n$-coherent if any subgroup of $G$ that is of  type $FP_n$ is of type $FP_{n+1}$.  Note that, for $n > 1$, $n$-coherence does not imply, nor is implied, by homological $n$-coherence.

 The following corollary follows from the above theorem. The case $n_0 = 1$ was proved earlier  in \cite{K-W}.

\begin{corollary} \label{cor1}  Let $G = K \rtimes \Gamma$, where $\Gamma$ is a finitely generated free non-cyclic  group, $K$ is of type $F_{n_0}$ (resp. $FP_{n_0}$) for some $n_0 \geq 1$. 
 Suppose that $\phi \colon K \to \R$ is a discrete character  with kernel $Ker(\phi) = N$ of type $F_{n_0-1}$ (resp.  $FP_{n_0-1}$) but not of type $F_{n_0}$  (resp. $FP_{n_0}$) and  such that  $\phi$  extends to a real homomorphism of $G$. Then there exists a discrete homomorphism $\psi \colon G \to \mathbb{R}$ that extends $\phi$ such that $Ker (\psi)$ is of type $F_{n_0}$ (resp. $FP_{n_0}$)  but is not of type $F_{n_0 + 1}$ (resp. $FP_{n_0 + 1}$). In particular $G$ is not $n_0$-coherent (resp. homologically $n_0$-coherent).
\end{corollary}

\begin{remark}
1) There are versions of  Theorem \ref{Main1} and Corollary \ref{cor1},  where type $FP_m$ is substituted with type $FP_m(R)$,  $R$ is any commutative ring with $1$. It suffices that in the proof of Corollary \ref{cor1} we use $\Sigma^m( - , R)$ (some authors use the notation  $\Sigma_R^m( - , R)$) instead of $\Sigma^m( - , \mathbb{Z})$. In particular this applies for $R = \mathbb{Q}$.

2) The above implies that  there is a mixed version of the corollary : if  $G = K \rtimes \Gamma$, where $\Gamma$ is finitely generated free non-cyclic, $K$ is of type $F_{n_0}$,  $\phi \colon K \to \R$ is a discrete character, $Ker(\phi) = N$ of type $F_{n_0-1}$ but not of type $FP_{n_0}$ ( resp. $FP_{n_0}(\mathbb{R})$) then there is  a discrete character $\psi \colon G \to \mathbb{R}$ that extends $\phi$ such that $Ker (\psi)$ is of type $F_{n_0}$ but not of type $FP_{n_0+1}$ ( resp. $FP_{n_0+1}(\mathbb{R})$).

3) The map $\psi$ from Corollary \ref{cor1}   can be any of the maps described in Remark \ref{rem:cone}. 
\end{remark}

In Section \ref{section-cor1} we discuss some corollaries of the above results for extensions of generalised Thompson groups $F_{n, \infty}$, see Corollary \ref{Thompson}. The proof of Corollary \ref{Thompson} depends on the structure of the $\Sigma$-invariants for $F_{n,\infty}$.  For general $n \geq 2$ the invariant $\Sigma^1(F_{n, \infty})$ was calculated by Bieri, Neumann and Strebel in \cite{B-N-S}. The $\Sigma$-invariants of $F = F_{2,\infty}$  were calculated by Bieri, Geoghegan and the first author in \cite{B-G-K}. The classification of $\Sigma^m(F_{n, \infty})$ for arbitrary $m$ and $n$ was completed by Zaremsky in \cite{Z}.

Furthermore in Section \ref{section-cor1} we discuss some corollaries about extensions of  some groups with negative Euler characteristics or some specific 2-generator groups. In order to do so, we will need the notion of excessive homology of a sequence. 
Suppose that $ 1 \to K \to G \to \Gamma \to 1$ is a short exact sequence of finitely generated groups. Using the terminology of \cite{K-W} we say that the {\em excessive homology} of this short exact sequence is $\operatorname{Ker}(H_1(G, \mathbb{R}) \to  H_1(\Gamma, \mathbb{R}))$, and we say that the sequence has excessive homology if that space has positive dimension.
 The following corollary generalizes the incoherence of a (non-cyclic f.g. free)-by-(non-cyclic  f.g. free) group with excessive homology \cite[Theorem 4.6]{K-W}.

\begin{corollary} \label{general}
Let $ 1 \to K \to G \to \Gamma \to 1$ be a short exact sequence  with excessive homology and infinite $\Gamma^{ ab}$  of groups of type $F$ such that  both $K$ and $\Gamma$ have negative Euler characteristic  $\chi(\Gamma) < 0, \chi(K) < 0$ and $cd(G) \leq 3$. 
Then $G$ is incoherent.
\end{corollary}

 In what follows, the term {\em surface group of genus $g$} will refer to the fundamental group of a connected, closed, orientable surface of genus $g > 0$. 
We use some ideas from the proof of Corollary \ref{general} to prove the following result.

\begin{proposition} \label{Sigma2} Let $G$ be a  (f.g. free or surface)-by-(f.g. free or surface) group with Euler characteristic $\chi(G) \not= 0$. Then $\Sigma^2(G, \mathbb{Z}) = \emptyset$, in particular $\Sigma^2(G) = \emptyset$.
\end{proposition}

  The following corollary generalizes the first part of \cite[Theorem 6.1]{K-W}, where $K$ is a free group of rank 2. Recall that a group $K$ algebraically fibers if there is an epimorphism $K \to \mathbb{Z}$ with a finitely generated kernel.
 
 \begin{corollary}  \label{gen} Let $G = K \rtimes \Gamma$ be a short exact sequence of groups such that $K$ is 2-generated, $K$ has a non abelian quotient $T = \mathbb{Z} \rtimes \mathbb{Z}_2$ and $K$ does not algebraically fiber. Suppose that $\Gamma$ is a finitely generated non-cyclic free group. Then $G$ is incoherent.
\end{corollary}

We want to address now the issue of applying iteratively Theorem \ref{Main1} to the case where a group $G$ is obtained as result of successive extensions. Recall that a {\em subnormal filtration of length $m$} of $G$ is a sequence $1  = G_0 \unlhd G_1 \unlhd \ldots \unlhd G_m = G$;  such a filtration is called {\em normal} if moreover each $G_i$ is normal in $G$. 
There is a challenge in applying iteratively Theorem \ref{Main1} coming from the fact that it may be unclear whether, at each step, the algebraic fibration of kernel $FP_{n_{0}-1}$ extends or not to the next stage; contrary to the case of length $2$, where $n_0 = 1$, this is not determined by excessive homology alone but requires a specific cohomology class (or at best one of a family of), determined in the previous iteration of Theorem \ref{Main1}, to extend. In practice, it may be hard to write down the optimal assumption for this to happen, so we limit ourselves to discuss some sufficient conditions for the existence of the extensions.
 In Section \ref{proofs-thm} we will show the following.

\begin{theorem} \label{thm:poly1} 
Let $s \geq 0$ and $m \geq 2$ be integers and $G$ be a group with a subnormal filtration $1  = G_0 \unlhd G_1 \unlhd \ldots \unlhd G_m = G$ such that 

a) there is a discrete character $\widetilde{\alpha}_j: G \to \mathbb{R}$ such that $\widetilde{\alpha}_j(G_j) = 0$ but $\widetilde{\alpha}_j(G_{j+1}) \not= 0$ for $0 \leq j \leq m-1$;

b) $G_j$ has type $FP_{ s+ j}$ for $ 1 \leq j \leq m-1$, $G_m$ has type $FP_{m+ s - 1}$.

 Then for every discrete character $\phi : G_1 \to \mathbb{R}$ that extends to a real character of $G$ and $Ker(\phi)$ has type $FP_{s}$, there is a discrete character $\psi : G \to \mathbb{R}$ that extends $\phi$ such that $Ker(\psi)$ has type $FP_{s + m - 1}$.
 
  Furthermore, if the filtration is normal, the assumption a) is equivalent to 

a') the sequence $1 \to G_{j+1}/ G_j  \to G/ G_j \to G/ G_{j+1} \to 1$ has excessive homology
for  $ 0 \leq j \leq m-1$.
 \end{theorem}

\begin{remark} 
 There are versions of  Theorem \ref{thm:poly1}  where type $FP_k$ is substituted with type $FP_k(\mathbb{Q})$, and other versions where type $FP_k$ is substituted with type $F_k$
\end{remark}

The main applications of Theorem \ref{thm:poly1} are discussed in Section \ref{pfs}, and concern the study of higher algebraic fibering of poly-free and poly-surface groups, and their role in the study of coherence. (Here, as described in Section \ref{pfs}, we require the filtrations to be subnormal only,  and furthermore we use the term {\em poly-free groups} to refer to what some author refer to as {\em poly-finitely generated free}.)  In particular, Hillman asked in \cite{Hi15} about coherence of poly-surface groups of length $2$, and whenever the sequence $1 \to K \to G \to \Gamma \to 1$ has excessive homology, the algebraic fibration $\phi \colon G \to \Z$ has kernel that is not finitely presented. Theorem \ref{thm:poly1} allows to obtain a similar result for higher coherence. 

The first result is the following.

\begin{theorem} \label{thm:poly} Let $G$ be a poly-free or a poly-surface group of length $m$ with subnormal filtration $1  = G_0 \unlhd G_1 \unlhd \ldots \unlhd G_m = G$ and assume that $H_1(G;\Z) \cong \oplus_{i=0}^{m-1}H_1(G_{i+1}/G_i;\Z)$; then $G$ admits an algebraic fibration $\psi \colon G \to \Z$ such that $M = \mbox{ker } \psi$ is $F_{m-1}$. Furthermore, if $\chi(G) \neq 0$, $M$ is not $FP_m$.  \end{theorem}

We use Theorems \ref{thm:poly1} and  \ref{thm:poly} to obtain some results on the finiteness of subgroups of some families of poly-free groups, including the pure braid group $P(m+1)$ (and the closely related upper McCool group $P\Sigma^{+}_{m+1}$), as well as some polysurface groups that arise as fundamental groups of complex projective varieties, building on work by Llosa Isenrich and Py in \cite{LIP}.  We have the following:

\begin{corollary} \label{cor:pbg} Let $K(m+1)$ be the pure mapping class group of a sphere with $m+2$ punctures, and let $P(m+1)$ be the pure braid group on $m+1$ strands. Then for all $0 \leq n \leq m-2$ these groups contain an algebraic fiber of type $F_{n}$ but not $FP_{n+1}$. Furthermore, $P(m+1)$ contains an algebraic fiber of type $F_{\infty}$. \end{corollary} 

\begin{corollary} \label{cor:itko}
For all $m \geq 0$ there exists a complex projective variety $X(m)$ with 
$\operatorname{dim}_{\C}(X(m)) = m$ whose fundamental group $\Pi(m)$ is a poly-surface group of length $m$ and nonzero Euler characteristic which admits, for all $0 \leq n \leq m-1$, a cocylic subgroup of type $F_n$ but not $FP_{n+1}$. \end{corollary}

In Section \ref{pfs} we compare the finiteness results of Corollaries \ref{cor:pbg} and \ref{cor:itko} with those of virtually RFRS groups developed by Fisher and Kielak \cite{Fisher}, \cite{Kielak}. 

\subsection*{Acknowledgement} We would like to thank Robert Kropholler, Claudio Llosa Isenrich, Matt Zaremsky and the referee for many useful comments.  The first  author was partially supported by Bolsa de produtividade em pesquisa CNPq 305457/2021-7 and Projeto tem\'atico FAPESP 18/23690-6. The second author was partially supported by the Simons Foundation Collaboration Grant For Mathematicians 524230.

\section{Preliminaries} \label{prel}

 The BNSR-invariants $\Sigma^n(G)$ and $\Sigma^n(G, \mathbb{Z})$ were defined in a sequence of papers. The most general version of $\Sigma^1(G)$ was defined in \cite{B-N-S}, $\Sigma^n(G, \mathbb{Z})$ was defined in \cite{B-R} and the homotopical version $\Sigma^n(G)$ was first considered in the PhD thesis \cite{Renzthesis}. All these invariants are subsets of the character sphere $S(G) = Hom(G, \mathbb{R}) \setminus \{ 0 \} / \sim$ where for $\chi_1, \chi_2 \in  Hom(G, \mathbb{R}) \setminus \{ 0 \}$ we have $\chi_1 \sim \chi_2$ if  there is  a positive real number $r$ such that $\chi_2 = r \chi_1$. The image of $\chi$ in $S(G)$ is denoted by $[\chi]$. An element $\chi \in  Hom(G, \mathbb{R}) \setminus \{ 0 \}$ is called a character.
 
 By definition  $$\Sigma^n(G, \mathbb{Z}) = \{[\chi] \in S(G) \mid \mathbb{Z} \textrm{ is of type $FP_n$ as $\mathbb{Z} G_{\chi}$-module}\},$$
 where  $G_{\chi} = \{ g \in G \ | \ \chi(g) \geq 0 \}$.
 If $G$ is not of type $FP_n$ then $\Sigma^n(G, \mathbb{Z}) = \emptyset$.
 
 Suppose that $G$ is of type $F_n$ and $\Gamma_n$ is the $n$-skeleton of the universal cover of some $K(G,1)$, where the $K(G,1)$ has finite $n$-skeleton. For simplicity we can assume that $K(G,1)$ has just one vertex and that the set of vertices of $\Gamma_n$ is precisely $G$. Thus $\Gamma = \Gamma_1$ is the Cayley complex with respect to some finite generating set of $G$.  Let $\chi$ be a character of $G$. Then $(\Gamma_n)_{\chi}$ denotes the subcomplex of $\Gamma_n$ spanned by the vertices in $G_{\chi}$. By definition
 $$
\Sigma^1(G) = \{ [\chi] \in S(G) \mid \Gamma_{\chi} \hbox{ is a connected graph}\}
$$ 
and for $n \geq 2$
$$
\Sigma^n(G) = \{ [\chi] \in S(G) \mid (\Gamma_n)_{\chi} \hbox{ is } (n-1)\hbox{-connected for some choice of } \Gamma_n \}.
$$
The above definition looks artificial as  for $n \geq 2$ it depends on the choice of $\Gamma_n$ but independantly from the choice of $\Gamma_n$ we have that $[\chi] \in \Sigma^n(G)$ if and only if there is $d < 0$ such that the map $\pi_1((\Gamma_n)_{\chi}) \to \pi_1((\Gamma_n)_{\chi \geq d})$ has trivial image, where $(\Gamma_n)_{\chi \geq d}$ denotes the subcomplex of $\Gamma_n$ spanned by the vertices in $G_{\chi \geq d} = \{ g \in G \ | \ \chi(g) \geq d \}$.
 
  The importance of the $\Sigma$-invariants lie in the following result.

\begin{theorem} \label{criterion} 
a) \cite{B-R} Let $G$ be a group of type $FP_m$ and $H$ be a subgroup of $G$
containing the commutator. Then $H$ is of type $FP_m$ if and only if
$S(G, H) = \{[\chi] \in S(G) \ | \ \chi(H) = 0 \} \subseteq  \Sigma^m(G, \mathbb{Z}).$

b) Let $G$ be a group of type $F_m$ and $H$ be a subgroup of $G$
containing the commutator. Then $H$ is of type $F_m$ if and only if
$S(G, H) = \{[\chi] \in S(G) \ | \ \chi(H) = 0 \} \subseteq  \Sigma^m(G).$
\end{theorem}

We state several results that would be used in the proofs of our main results.

\begin{theorem} \cite{B-R} \label{open} Let $G$ be a finitely generated group. Then $\Sigma^m(G, \mathbb{Z})$ is open in $S(G)$. If $G$ is of type $F_m$ then $\Sigma^m(G)$ is open in $S(G)$.
\end{theorem}

We observe that if we use the notation $[\chi]$ it presumes that $\chi$ is non-trivial. 

\begin{theorem} \cite{S1} \label{SS} Suppose that $G$ acts on a tree $T$ such that $T$ is
finite modulo the action of $G$. 

i) If $n \geq 1$, if $[\chi_v] \in \Sigma^n
(G_v , \mathbb{Z} )$
 for all vertices $v$ of $T$, and if  $[\chi_e] \in \Sigma^{n-1}(G_e, \mathbb{Z})$
  for all edges $e$ of $T$,  then $[\chi] \in \Sigma^n
(G, \mathbb{Z}).$

ii) If $n \geq 0$, if $[\chi] \in \Sigma^n
(G, \mathbb{Z} )$, and if $[\chi_e] \in \Sigma^n (G_e, \mathbb{Z} )$ for all edges $e$ of $T$ , then $[\chi_v] \in \Sigma^n(G_v , \mathbb{Z})$ for all vertices $v$ of $T$.

iii) If $n \geq 1$, if $[\chi] \in \Sigma^n(G, \mathbb{Z})$, and if $[\chi_v] \in \Sigma^{ n-1} (G_v , \mathbb{Z} )$ for all vertices $v$ of
$T$ and $\chi_e \not= 0$ for all edges $e$ of $T$, then $[\chi_e] \in \Sigma^{n-1} (G_e, \mathbb{Z} )$ for all edges $e$ of $T$.
\end{theorem}

 A more general homotopical version of item i) from the above theorem is stated below. In particular  there is a homotopical version of item i) for $n = 2$. We think of a tree as a $2$-complex without 2-cells. 

\begin{theorem} \cite{M}
Let a group $G$ act on a $l$-connected $G$-finite $2$-complex $X$. If $\chi :  G \to \mathbb{R}$ is a homomorphism such that $0 \not= \chi |_{G_{\sigma}},$  $ [\chi |_{G_{\sigma}}]  \in \Sigma^{ 2 - dim(\sigma)}(G_{\sigma})$ for all cells $\sigma$ of $X$, then $[\chi] \in \Sigma^2(G)$. 
\end{theorem}

Since in general for $ n \geq 2$ we have  $\Sigma^n(G) = \Sigma^n(G, \mathbb{Z}) \cap \Sigma^2(G)$ and $\Sigma^1(G) = \Sigma^1(G, \mathbb{Z})$,  the above two theorems imply

\begin{corollary} \label{SSS}
Suppose that $G$ acts on a tree $T$ such that $T$ is
finite modulo the action of $G$. Suppose further that the restriction $\chi_{\sigma} : G_{\sigma} \to \mathbb{R}$
of $\chi$  to the stabilizer $G_{\sigma}$ of a vertex or an edge $\sigma$ of $T$ is non-zero.

 If $n \geq 1$, if $[\chi_v] \in \Sigma^n
(G_v )$
 for all vertices $v$ of $T$, and if  $[\chi_e] \in \Sigma^{n-1}(G_e)$
  for all edges $e$ of $T$,  then $[\chi] \in \Sigma^n
(G).$
\end{corollary}

In the above corollary if $n = 1$, $\Sigma^0(G_e)$ denotes $S(G_e)$. 

The homotopical part of the following result was proved by Kuckuck in \cite{K} with geometric methods. It   was used  in the proof of special cases of the $n-(n+1)-(n+2)$-Conjecture. In \cite{K-L} the first author and Lima proved a homological version of Kuckuck's result, using only algebraic methods including spectral sequences.

\begin{proposition} \cite{K}, \cite{K-L} \label{Kuck}  Let $n \geq 1$ be a natural number, $A \hookrightarrow B \twoheadrightarrow C$ a short exact sequence of groups with  $A$ of type $F_n$ ( resp. of type $FP_n$) and $C$ of type $F_{n+1}$ (resp. of type $FP_{n+1}$). Assume there is another short exact sequence of groups $A \hookrightarrow B_0 \twoheadrightarrow C_0$ with $B_0$ of type $F_{n+1}$ (resp. of type $FP_{n+1}$) and that there is a group homomorphism $\theta: B_0 \rightarrow B$ such that $\theta|_A = id_A$,  i.e. there is a commutative diagram of homomorphisms of groups  $$\xymatrix{A \ \ar@{^{(}->}[r] \ar[d]_{id_A} & B_0 \ar@{->>}[r]^{\pi_0} \ar[d]_{\theta} & C_0 \ar@{.>}[d]^{\nu} \\ A \ \ar@{^{(}->}[r] & B \ar@{->>}[r]^{\pi} & C}$$  Then $B$ is of type $F_{n+1}$ ( resp. of type $FP_{n+1}$).
\end{proposition}

It is interesting to note that in the above proposition $\theta$ is not required to be surjective but we will apply it in the proof of our main results in the case when $\theta$ is surjective.

\section{Proofs} \label{proofs-thm}

The proof of \cite[Theorem 1]{F-V} uses significantly the following lemma.

\begin{lemma}  \cite[Lemma ~2.2]{F-V} \label{l-2.2}
Let $\Pi = \Pi_1 *_K \Pi_2$ be a free product with amalgamation of two finitely
generated groups along a finitely generated subgroup $K$. Assume that  $\Pi_2 = K \rtimes \mathbb{Z} = \langle K, s \ | \ sks^{ -1}  = f (k) \rangle$ for some automorphism $f : K \to K$.
Let $[\chi] \in S(\Pi)$ be a character whose restrictions satisfy the conditions $[\chi_1] \in \Sigma^1(\Pi_1)$, $\chi(K) \not= 0$
 and $\chi_2(s) = 0$, where $\chi_1 = \chi |_{\Pi_1}$, $\chi_2 = \chi |_{\Pi_2}$. Then $[\chi] \in \Sigma^1(\Pi)$. 
\end{lemma} 

The following result generalizes Lemma \ref{l-2.2}.

\begin{lemma} \label{gen-lemma} 
Let $\Pi = \Pi_1 *_K \Pi_2$ and $\Pi_2 = K \rtimes \mathbb{Z} = \langle K, s \ | \ sks^{-1} = f(k) \hbox{ for } k \in K \rangle$ for some $f \in Aut(K)$. Let $\chi : \Pi \to \mathbb{R}$ be a character such that $\chi |_{\Pi_1} \not= 0$ and $\chi |_K \not= 0$.

a) Suppose that $\Pi_1, \Pi_2$ are of type $F_n$ and $K$ is of type $F_{n-1}$,  $[\chi |_{\Pi_1}] \in \Sigma^n(\Pi_1)$, $[\chi |_K] \in \Sigma^{ n-1} (K)$. Then $[\chi] \in \Sigma^n(\Pi)$.

b) Suppose that $\Pi_1, \Pi_2$ are of type  $FP_n$ and $K$ is of type $FP_{n-1}$,  $[\chi |_{\Pi_1}] \in \Sigma^n(\Pi_1, \mathbb{Z})$, $[\chi |_K] \in \Sigma^{ n-1} (K, \mathbb{Z})$. Then $[\chi] \in \Sigma^n(\Pi, \mathbb{Z})$.
\end{lemma}

\begin{proof} Note that $\Pi = \langle \Pi_1, s \ | \  sks^{-1} = f(k), k \in K \rangle$. Thus $\Pi$ is an HNN extension with stable  letter $s$, base group $\Pi_1$ and associated subgroups $K, f(K) = K$. Then $\Pi$ acts on  the associated Bass-Serre tree $T$  with vertex stabilizers conjugates of $\Pi_1$ and edge stabilizers conjugates of $K$. 
Then
a) is a particular case of Corollary \ref{SSS} and
b) is a particular case of  Theorem \ref{SS}.
\end{proof}

\begin{proof}[Proof of Theorem \ref{Main1}] In section \ref{prel} we defined a character to be non-trivial, thus $\phi \not= 0$.

We start considering the homotopical version first. 
We follow the recipe of  the proof of Theorem 1 from \cite{F-V} substituting $\Sigma^1(G)$ with $\Sigma^{n_0}(G)$ and substituting Lemma \ref{l-2.2} with Lemma \ref{gen-lemma} a). In the final stage of the proof we need to extend the argument since a quotient of a group of type $F_{n_0}$ is not necessary of type $F_{n_0}$.

For a finitely generated group $G$ we denote by $G^{ ab}$ the abelianization of $G$ and by $Tor$ the maximal finite subgroup of $G^{ab}$. For a subgroup $K$ of $G$ the map $K \to  G^{ ab}/ Tor$  is the composition of the inclusion $K \to G$ and the canonical projection $G \to G^{ab}/ Tor$.

As in the proof of \cite[Theorem 1]{F-V} there is a generating set $h_1, \ldots, h_m, g_1, \ldots, g_r$ of $\Gamma$ that is ``adjusted to ab'' i.e. $\langle \overline{h_1}, \ldots, \overline{h_m} \rangle \simeq \mathbb{Z}^m \simeq \Gamma^{ ab} / Tor$, where $\overline{h_1}, \ldots, \overline{h_m} $ are the images of $h_1, \ldots, h_m$ in $\Gamma^{ ab} / Tor$ and $g_1, \ldots, g_r \subset Ker ( \Gamma \to \Gamma^{ ab} / Tor)$. Set $n = m+ r$.
As in the proof of \cite[Thm. ~1]{F-V}  there is a commutative diagram where the lines are short exact sequences of groups
$$\xymatrix{K \ \ar@{^{(}->}[r] \ar[d]_{id_K} & \Pi \ar@{->>}[r]^{} \ar@{->>}[d]_{\pi} & F_n \ar@{->>}[d]^{} \\ K \ \ar@{^{(}->}[r] & G \ar@{->>}[r]^{} & \Gamma}$$
where $F_n$ is the free group with a free basis $s_1, \ldots, s_n$,  and the map $F_n \to \Gamma$ sends $s_1, \ldots, s_n$ to $h_1, \ldots, h_m, g_1, \ldots, g_r$ and the map $\pi |_K : K \to K$ is the identity. Here $\Pi = \Pi_1 *_K \Pi_2 *_K \ldots *_K \Pi_n$ and each $\Pi_i = K \rtimes \langle s_i \rangle$, $\langle s_i \rangle \simeq \mathbb{Z}$. 
Let $\widetilde{\alpha} : G \to \mathbb{R}$ be a discrete character (i.e. non-zero homomorphism with cyclic image) such that $K \subseteq Ker(\widetilde{\alpha})$. Let $\alpha =  \widetilde{\alpha} \circ \pi : \Pi \to \mathbb{R}$  and  $\alpha |_{\Pi_i} = \alpha_i$.  Without loss of generality, permuting if necessary the elements $h_1, \ldots, h_m$, we can assume that $\alpha_1 \not= 0$. Then $K = Ker (\alpha_1)$.

Note that in the statement of the result we assumed that there is a character of $G$ that extends $\phi$. This is equivalent to the fact that the image of $K$ in the abelianization $G/ G'$ is infinite. This implies the existence of a {\em discrete} character  $c : G \to \mathbb{R}$  whose restriction to $K$ is the character $\phi : K \to \mathbb{R}$. 
 
Define $\gamma = c \circ \pi : \Pi \to \mathbb{R}$. Set $\gamma_i = \gamma |_{\Pi_i}$ and  $\beta_i : = \alpha_i + \mu \gamma_i : \Pi_i \to \mathbb{R}$ for some rational number $\mu > 0$, thus $\beta_i$ is a discrete character. Note that $\beta_i |_K = \mu \gamma_i |_{K} = \mu  \phi$ and define the discrete character $\beta : \Pi \to \mathbb{R}$ given by $\beta |_{\Pi_i} = \beta_i$.

Since $K$ is  of type $F_{n_0}$ and $Ker(\alpha_1) = K$, then $[\alpha_1] \in \Sigma^{ n_0}(\Pi_1)$. By Theorem \ref{open} $\Sigma^{ n_0}(\Pi_1)$ is an open subset of $S(\Pi_1)$, hence for sufficiently small $\mu$ we have that $[\beta_1] \in \Sigma^{ n_0}(\Pi_1)$.

Recall that $\beta_i |_K  = \mu \phi$. Since $N$ is of type $F_{n_0-1}$ we have that $[\beta_i |_K] = [\phi]  \in \Sigma^{n_0-1}(K)$. Then by Lemma \ref{gen-lemma} applied $n-1$ times $[\beta] \in \Sigma^{n_0} (\Pi)$. The same argument shows that  $[- \beta] \in \Sigma^{n_0} (\Pi)$.
 Hence $Ker (\beta)$ is of type $F_{n_0}$ and there is a short exact sequence $1 \to Ker (\beta) \to \Pi \to \mathbb{Z} \to 1$.

Recall that $\beta = \alpha + \mu \gamma$ and that both $\alpha$ and $\gamma$ factor through $G$. Thus
$\beta$ induces a discrete character $b : G \to  \mathbb{R}$. Actually $b = \widetilde{\alpha} + \mu c$.  And there is an epimorphism $Ker(\beta) \to Ker(b)$, whose kernel is isomorphic to $T = Ker(\pi : \Pi \to G)$. Note that both $\Pi$ and $G$ are of type $F_{ n_0}$. In the original \cite[Theorem 1]{F-V} the proof finishes here, since a quotient of a finitely generated group is finitely generated. But since a quotient of a finitely presented group is not necessary finitely presented we need an additional argument here to deduce that $Ker(b)$ is of type $F_{n_0}$.

Now for the general case  i.e. $n_0 \geq 2$ we would apply Kuckuck's criterion i.e. the homotopical part of
 Proposition \ref{Kuck}.
Write $\overline{Ker (\beta)}$ for the image of $Ker (\beta)$ in $F_n$ and $\overline{Ker(b)}$ for the image of $Ker(b)$ in $\Gamma$. By construction $Ker(\beta) \cap K = N$ and $Ker(b) \cap K = N$. By assumption $N$ is of type $F_{n_0-1}$ and we have already proved that $Ker(\beta)$ is of type $F_{n_0}$.  By constuction $b(K) \not= 0$, hence $K. Ker (b) \not= Ker(b)$ and since $G/ Ker(b) \simeq \mathbb{Z}$ we deduce that $K. Ker(b)$ has finite index in $G$ and so $\overline{Ker(b)}$ has finite index in $\Gamma$. Since in the short exact sequence $1 \to K \to G \to \Gamma \to 1$ we have that $G$ and $K$ are of type $F_{n_0}$ ( here we need only that $K$ is $F_{n_0-1}$) we deduce that $\Gamma$ is of type $F_{n_0}$. Then $\overline{Ker(b)}$ is of type $F_{n_0}$. Then we can apply the Kuckuck's criterion i.e. the homotopical version of Proposition \ref{Kuck} for the commutative diagram 

$$\xymatrix{N = Ker(\beta) \cap K \ \ar@{^{(}->}[r] \ar[d]_{id_N} &  Ker (\beta) \ar@{->>}[r]^{} \ar@{->>}[d]_{\pi |_N} &  \overline{Ker(\beta)} \ar@{->>}[d]^{} \\ N = Ker(b) \cap K  \ \ar@{^{(}->}[r] & Ker (b) \ar@{->>}[r]^{} & \overline{Ker(b)}}$$
 to deduce that $Ker(b)$ is of type $F_{n_0}$. Recall that $b = \widetilde{\alpha} + \mu c$, hence $b |_K = \mu c|_K = \mu \phi$. Then we can set $\psi = \mu^{-1} b$, thus $\psi |_K  = \phi$.

Since the proof of Proposition \ref{Kuck} is complicated we  include an elementary proof of the final part of the above argument in the special case when $n_0 = 2$. We are not aware of simple arguments for $n_0 > 2$.
Let $T = Ker (\pi)$ and let $t  \in T$. Note that the conjugation action of $t$ on $K$ in $\Pi$ gives an automorphism of $K$ that coincides with the conjugation action of $\pi(t) = 1$ on $K$ in $G$ i.e. $[T, K] = 1$ in $\Pi$.
Since $G$ is finitely presented  there is a finite set $X = \{t_1, \ldots, t_s \} \subseteq  T$ that generates $T$ as a normal subgroup of $\Pi$.

By construction  there is a short exact sequence $ 1 \to T = \langle X \rangle^{\Pi} \to Ker (\beta) \to Ker (b) \to 1$. As well by construction $\beta(K) \not= 0$ and since $\Pi / Ker (\beta) \simeq \mathbb{Z}$ we deduce that $K \subsetneq K. Ker (\beta)$ and $ K. Ker (\beta)$   has finite index in $\Pi$. Let $a_1, \ldots, a_i$ be a transversal of $K. Ker(\beta)$ in $\Pi$. Since $K$ is normal in $\Pi$ we have that $\Pi = \cup_{1 \leq j \leq i}  a_j K. Ker (\beta) = \cup_{1 \leq j \leq i}  K a_j Ker (\beta)$. Since $[T, K] = 1$ we have that $[X,K] = 1$, hence
$T = \langle X \rangle^{\Pi} = \langle \cup_{1 \leq j \leq i} X^{K a_j} \rangle^{Ker(\beta)} =
\langle \cup_{1 \leq j \leq i} X^{a_j} \rangle^{Ker(\beta)} $ is a normal subgroup of $Ker(\beta)$ that is finitely generated as a normal subgroup of $Ker(\beta)$. Then $Ker(b)$ is finitely presented. A homological version of this argument works too i.e. if $K, G, \Pi$ are of type $FP_2$  and we know that $Ker(\beta)$ is of type $FP_2$ we can deduce that $Ker(b)$ is of type $FP_2$.

The proof of the homological version of Theorem \ref{Main1} is similar i.e. it is a homological version of the above proof, where we use property $FP_{k}$ instead of $F_{k}$ and we use $\Sigma^{k}( -, \mathbb{Z})$ instead of $\Sigma^k( - )$  for $k = n_0$ and $k = n_0 -1$. Instead of Lemma \ref{gen-lemma} a) we use Lemma \ref{gen-lemma} b). As well  we use that by  Theorem \ref{open} $\Sigma^k (G, \mathbb{Z})$ is open in $S(G)$. With the above modifications we get that $Ker(\beta)$ is of type $FP_{n_0}$. In the final stage of the proof  we use the homological version  of Proposition \ref{Kuck} to deduce that $Ker(b)$ is of type $FP_{n_0}$.

Finally the case when $n_0 = \infty$ is proved as above by using the invariants $\Sigma^{\infty}( - ) = \cap_m \Sigma^m ( - )$ and  $\Sigma^{\infty}( -, \mathbb{Z} ) = \cap_m \Sigma^m ( - , \mathbb{Z})$.
\end{proof}

\begin{remark}  \label{rem:cone} The map $\psi$ from  Theorem \ref{Main1}  can be any map of the type $\mu^{-1} \widetilde{\alpha} + c$, where $\widetilde{\alpha} : G \to \R$ is any discrete character (i.e. non-zero homomorphism with infinite cyclic image) such that $K \subseteq Ker (\widetilde{\alpha})$, $c : G \to \R$ is any discrete character that extends $\phi$ and $\mu$ is any sufficiently small positive rational number, precisely how small  depends on $\widetilde{\alpha}$ and $c$. Since $\widetilde{\alpha}$ can be substituted with $ - \widetilde{\alpha}$ we can remove the restriction that $\mu$ is positive and consider $\mu$ as a non-zero rational number with sufficiently small absolute value, say $|\mu| < \mu_0$. Fixing $c$ and $\widetilde{\alpha}$ the set  $\mu^{-1} \widetilde{\alpha} + c$ described above corresponds to the rational points in a specific open cone inside $\mathbb{R} \widetilde{\alpha} + \mathbb{R} c$ sitting above $c$. The open cone is given by the condition $|\mu^{-1}| > \mu_0^{-1}$. \end{remark}

\begin{proof}[Proof of Theorem \ref{thm:poly1}] We induct on $m \geq 2$. The case $m = 2$ is Theorem \ref{Main1}.
Suppose the result holds for $m-1$.
Let $c : G \to \mathbb{R}$ be a discrete character that extends $\phi$. Consider  a discrete character $\widetilde{\alpha} = \widetilde{\alpha}_1 : G \to \mathbb{R}$ such that $\widetilde{\alpha}(G_1) = 0$ and $\widetilde{\alpha} (G_2) \not= 0$. Let $\mu > 0$ be a rational number. Note that $\mu^{-1} \widetilde{\alpha} + c : G \to \mathbb{R}$ and  $\mu^{-1} \widetilde{\alpha} + c |_{G_2} : G_2 \to \mathbb{R}$ are discrete characters that extend $\phi$. By the remark  above for $\mu$ sufficiently small  we have that
$Ker(\mu^{-1} \widetilde{\alpha} + c |_{G_2})$ is of type $FP_{s+1}$.

Consider now  the subnormal filtration $1  = G_0 \unlhd G_2 \unlhd G_3 \unlhd \ldots \unlhd G_m = G$. By construction $\mu^{-1} \widetilde{\alpha} + c |_{G_2} : G_2 \to \mathbb{R}$ is a discrete character that extends to the discrete character $\mu^{-1} \widetilde{\alpha} + c : G \to \mathbb{R}$. Thus by induction there is a discrete character $\psi : G \to \mathbb{R}$ that extends  $\mu^{-1} \widetilde{\alpha} + c |_{G_2}$ and $Ker(\psi)$ has type $FP_{s+1+ m-2} = FP_{s + m-1}$.

To prove the equivalence of a) with a') in the case where the filtration is normal, consider the following commutative diagram, whose existence is guranteed by the normality assumption:
\[ \xymatrix{
& G_j \ar@{^{(}->}[d]  \ar@{^{(}->}[r]   & G_{j+1} \ar@{^{(}->}[d] \\
 & G \ar@{->>}[d]\ar[r]^{\cong} & G \ar@{->>}[d] \\
G_{j+1}/G_{j} \ar@{^{(}->}[r]  & G/G_{j}  \ar@{->>}[r]   &
G/G_{j+1}  } \]
A discrete character $\widetilde{\alpha_{j}} : G \to \R$ as in a)  descends to a character of $G/G_{j}$, but is not a pull-back of a character of $G/G_{j+1}$, so the sequence in the bottom row has excessive homology. Conversely, if that sequence has excessive homology, there is a character of $G/G_{j}$, hence of $G$, vanishing on $G_{j}$ but not on $G_{j+1}$. \end{proof}

\begin{proof}[Proof of Corollary \ref{cor1}]

We start with the proof of the homological version of Corollary \ref{cor1}. We define $M$ as in the proof of Theorem \ref{Main1} for $\Gamma = F_n$ and $\pi$ the identity map. Thus $M = Ker(\beta) = Ker(b)$ is a normal subgroup of $G$, $G/ M \simeq \mathbb{Z}$ and $M$ is of type $FP_{n_0}$. By construction  $N = Ker (\beta |_{K})$, recall that $\beta |_K = \mu \phi$, where $\mu$ is a suffciently small positive rational number. Since $N$ is not of type $FP_{n_0}$ then at least one of $[\beta |_{K}] = [\phi], [ - \beta |_{K}] = [- \phi]$ is not an element of $\Sigma^{n_0} (K, \mathbb{Z})$. By changing $\beta$ to $ - \beta$ if necessary we can assume that $[\beta |_{K}] = [\phi] \not\in \Sigma^{n_0} (K, \mathbb{Z})$. 

Assume now that $M$ is of type $FP_{n_0 + 1}$, hence $[\beta] \in \Sigma^{ n_0 + 1} (G)$. Recall that $G = \Pi_1 *_K \Pi_2 *_K \ldots *_K \Pi_n$. For $i \geq 2$ we think of $G_i =  \Pi_1 *_K \Pi_2 *_K \ldots *_K \Pi_i$ as an HNN extension with a base group $G_{i-1}$, associated groups that both coincide with $K$ and stable letter $s_i$. By construction $\beta |_{\Pi_1} \in \Sigma^{n_0} (\Pi_1, \mathbb{Z})$ and as in the proof of Theorem \ref{Main1} we have that $[\beta |_{G_i}] \in \Sigma^{ n_0} (G_i, \mathbb{Z})$ for every $i \geq 1$.
By Theorem \ref{SS}, iii) applied for the Bass-Serre tree associated to the HNN extension $G_{n-1} *_K \Pi_n = \langle G_{n-1}, s_n \ | \ s_n k s_n^{ -1} = f_n(k),k \in K \rangle$ with base group $G_{n-1}$, associated subgroups equal to $K$ and stable letter $s_n$,  and using  that $[\beta] \in \Sigma^{ n_0 + 1} (G,  \mathbb{Z} )$ and $[\beta |_{G_{n-1}}] \in \Sigma^{ n_0 } (G, \mathbb{Z})$ we deduce that $ [\beta |_K] \in \Sigma^{ n_0} ( K, \mathbb{Z} )$, a contradiction.

Suppose now that the assumptions of the homotopical version of Corollary \ref{cor1} hold and $n_0 = 1$. By the homological version  there is a character $\psi$ of $G$ such that $M = Ker(\psi)$ is $FP_1$ ( i.e. finitely generated) but not $FP_2$. In particular $M$ is not finitely presented.

Suppose now that the assumptions of the homotopical version of Corollary \ref{cor1} hold and $n_0 \geq 2$. Then we define $\beta$ as above and by the homotopical version of Theorem \ref{Main1}, using that $N$ is finitely generated, $K$ and $G$ are finitely presented,  we deduce that $M = Ker(\beta)$ is finitely presented. By the above argument $M$ is of type $FP_{n_0}$ but is not of type $FP_{n_0 + 1}$. Thus $M$ is of type $F_{n_0}$ but is not of type $F_{n_0 + 1}$. \end{proof}

\section{Some corollaries} \label{section-cor1}

\subsection{ Generalised Thompson groups}
Consider the generalised Thompson group $F_{n, \infty}$ given by the infinite presentation
$$F_{n, \infty} = \langle x_0, x_1, \ldots, x_s, \ldots \ | \  x_i^{ x_j} = x_{i + n-1} \hbox{ for }  i > j \geq 0 \rangle.$$
Some authors use the notation $F_n$ but we avoid it since  it could be confused with a free group of rank $n$ or with the notation for homotopical type $F_n$. Note that $F = F_{2,\infty}$ is one of the classical Thompson groups. It is known that $F_{n, \infty}$ is of type $F_{\infty}$ i.e. finitely presented and of type $FP_{\infty}$ \cite{Brown} but has infinite cohomological dimension since contains free abelian groups of arbitrary rank.

For general $n \geq 2$ the invariant $\Sigma^1(F_{n, \infty})$ was calculated in \cite{B-N-S} and it was shown that the complement of $\Sigma^1(F_{n, \infty})$ in $S(F_{n, \infty})$ has two non-antipodal discrete points, we denote here by $[\chi_1]$ and $[\chi_2]$. The $\Sigma$-invariants of $F = F_{2,\infty}$ in the case $n = 2$ were calculated by Bieri, Geoghegan and the first author in \cite{B-G-K}. The invariant $\Sigma^2(F_{n, \infty})$ was calculated by the first author in \cite{K} using combinatorial methods i.e. splitting the original group in different ways as free product with amalgamation.

\begin{theorem}  \cite{B-G-K} For $F = F_{2, \infty}$ and $m \geq 2$ we have $\Sigma^m(F)^c : = S(F) \setminus \Sigma^m(F) = conv \{ [\chi_1], [\chi_2] \}$.
\end{theorem}

\begin{theorem} \cite{K} For any $n \geq 2$ we have $\Sigma^2(F_{n, \infty}, \mathbb{Z}) = \Sigma^2(F_{n, \infty}) =  conv \{ [\chi_1], [\chi_2] \}$
\end{theorem}
Here $ conv \{ [\chi_1], [\chi_2] \}$ denotes the smalles arc with end points $[\chi_1], [\chi_2]$ in the character sphere $S(F_{n, \infty})$.

The classification of $\Sigma^m(F_{n, \infty})$ for arbitrary $m$ and $n$ was completed by Zaremsky in \cite{Z}, using geometric methods like Morse Theory and actions on a  $CAT(0)$-cube complex.

\begin{theorem} \cite{Z} \label{thmZ}  For any $n,m \geq 2$ we have $\Sigma^m(F_{n, \infty}) = \Sigma^2(F_{n, \infty})$.
\end{theorem}

\begin{corollary} For any $n,m \geq 2$ we have $\Sigma^m(F_{n, \infty}, \mathbb{Z}) = \Sigma^2(F_{n, \infty})$.
\end{corollary}

\begin{proof}
For a general group $G$ of type $F_m$, where $m \geq 2$, we have that $$\Sigma^m(G) \subseteq \Sigma^m(G, \mathbb{Z}) \subseteq \Sigma^2(G, \mathbb{Z})$$ Applying this for $G = F_{n, \infty}$ we have
$$\Sigma^m(F_{n, \infty}) \subseteq \Sigma^m(F_{n, \infty}, \mathbb{Z}) \subseteq \Sigma^2(F_{n, \infty}, \mathbb{Z}) = \Sigma^2(F_{n, \infty}) = \Sigma^m(F_{n, \infty}).$$
\end{proof}

\begin{corollary} \label{Thomp} Let $\chi : F_{n, \infty} \to \mathbb{Z}$ be a character with $N = Ker (\chi)$. Then  

a) $N$ is finitely generated if and only if $[\chi] \notin  \{ [\chi_1], [\chi_2], [- \chi_1], [ - \chi_2]  \}$;

b) $N$ is
 finitely presented  if and only if $[\chi], [ - \chi] \in S(F_{n, \infty}) \setminus conv \{ [\chi_1], [\chi_2] \}$;
 
 c) $N$ is of type $F_{\infty}$ if and only if $N$ is finitely presented;
 
 d) $N$ is of type $FP_2$ if and only if $N$ is finitely presented;
 
 e) $N$ is of type $FP_m$ for some $m \geq 2$ if and only if $N$ is finitely presented;
 
In particular for every normal subgroup $S$ of $F_{n, \infty}$ such that $F_{n, \infty}/ S \simeq \mathbb{Z}^k$ for some $k \geq 2$, there exists a subgroup $N$ of $F_{n, \infty}$ such that $S \subseteq N $, $F_{n, \infty}/ N \simeq \mathbb{Z}$ and $N$ is of type $F_{\infty}$.
\end{corollary}

\begin{proof} By Theorem \ref{criterion}  $N$ is of type $F_m$ if and only if $[\chi], [- \chi] \in \Sigma^{m} (F_{n, \infty})$. By the above results $\Sigma^1(F_{n, \infty}) = S(F_{n, \infty}) \setminus \{ [\chi_1], [\chi_2] \}$ and for $ m \geq 2$ we have $\Sigma^{m} (F_{n, \infty}) = S(F_{n, \infty}) \setminus conv \{ [\chi_1], [\chi_2] \}$.

 By Theorem \ref{criterion}  $N$ is of type $FP_m$ if and only if $[\chi], [- \chi] \in \Sigma^{m} (F_{n, \infty}, \mathbb{Z})$. By the above results for $ m \geq 2$ we have $\Sigma^{m} (F_{n, \infty}, \mathbb{Z}) = S(F_{n, \infty}) \setminus conv \{ [\chi_1], [\chi_2] \}$.

 The final part follows from Theorem \ref{criterion} and the fact that $conv \{ [\chi_1], [\chi_2] \} \cup conv \{ [- \chi_1], [ - \chi_2] \}$ is strictly inside $S(G, N_0) = \{ [\chi] \in S(G) \ | \ \chi(N_0) = 0 \}$, where $N_0 = Ker (\chi_1) \cap Ker (\chi_2)$. Thus there exists a character  $\chi : F_{n, \infty} \to \mathbb{Z}$ such that  $[\chi], [ - \chi] \in S(F_{n, \infty}) \setminus conv \{ [\chi_1], [\chi_2] \}$ and $\chi(N_0) = 0$. We set $N = Ker (\chi)$.
\end{proof}

\begin{corollary} \label{Thompson} Suppose that $ 1 \to F_{n, \infty} \to G \to \Gamma \to 1$ is a short exact sequence of groups with $G$ of type $F_{n_0}$ (resp. $FP_{n_0}$). Suppose that one of the following conditions holds

a) the excessive homology is at least 2;

b) the excessive homology is exactly 1, $n_0 \geq 2$, $\chi : F_{n, \infty} \to \mathbb{Z}$ is the unique character (up to a sign) such that $\chi(K_0) = 0$, where $K_0 = Ker (F_{n, \infty} \to G^{ ab}/ Tor)$ and $[\chi] \notin conv \{ [\chi_1], [\chi_2] \} \cup conv \{ [- \chi_1], [ - \chi_2] \}$;

c) the excessive homology is exactly 1, $n_0 = 1$, $\chi : F_{n, \infty} \to \mathbb{Z}$ is the unique character (up to a sign) such that $\chi(K_0) = 0$, where $K_0 = Ker (F_{n, \infty} \to G^{ ab}/ Tor)$ and $[\chi] \notin  \{ [\chi_1], [\chi_2], [- \chi_1], [ - \chi_2] \}$;

Then  there is a normal  subgroup $M$ of $G$ such that $G/ M \simeq \mathbb{Z}$ and $M$ is of type $F_{n_0}$ (resp. $FP_{n_0}$). Furthermore if $G$ is of type $F_{\infty}$ (resp. $FP_{\infty}$)  then $M$ is of type $F_{\infty}$ (resp. $FP_{\infty}$).
\end{corollary}

\begin{proof} It is enough to show that   the conditions of  Theorem \ref{Main1}  applied for $K = F_{n, \infty}$ hold. Indeed for a) we can use the final part of  Corollary \ref{Thomp} applied for $S =  Ker (F_{n, \infty} \to G^{ ab}/ Tor)$. For b) we  can deduce from Theorem \ref{criterion}  and Theorem \ref{thmZ} that $[\chi] \in \Sigma^{\infty}(F_{n, \infty})$, hence $Ker (\chi)$ is of type $F_{\infty}$, in particular is of type $F_{n_0}$. For c) we  can deduce from Theorem \ref{criterion}  that $[\chi] \in \Sigma^{1}(F_{n, \infty})$, hence $Ker (\chi)$ is finitely generated.
\end{proof}

\subsection{Some corollaries that use the Novikov ring $ \widehat{\mathbb{Z} G}_{\chi}$}

Let $\chi : G \to \mathbb{R}$ be a character i.e. a non-trivial homomorphism. We start by recalling the definition of the Novikov ring   $\widehat{\mathbb{Z} G}_{\chi}$ that will be used in the proof of the following result. First write $\widehat{\mathbb{Z} G}$ for the set of all formal (possibly infinite) sums $\lambda = \sum_{g \in G} z_g g$, where $ z_g \in \mathbb{Z}$. By definition  $supp(\lambda) = \{ g \in G \ | \ z_g \not= 0 \}$ and
$\widehat{\mathbb{Z} G}_{\chi}$  is the set of elements $\lambda \in \widehat{\mathbb{Z} G}$ for which for every $r \in \mathbb{R}$ we have that the set $supp (\lambda) \cap \chi^{ -1} ((- \infty, r]) $ is finite for every $r \in \mathbb{R}$.

Let $G$ be a group of type $FP_3$ and $$\ldots \to \mathbb{Z} G^{k_3} \to \mathbb{Z} G^{k_2} \to \mathbb{Z} G^{k_1} \to \mathbb{Z} G^{k_0} \to \mathbb{Z} \to 0$$ be a free resolution of the trivial $\mathbb{Z} G$-module $\mathbb{Z}$. Define
$$
\chi_3(G) = max \{ k_0 - k_1 + k_2 - k_3 \}$$ where the maximum is over all possible free resolutions as above.

\begin{theorem} \label{incoherent} 
Let $ 1 \to K \to G \to \Gamma \to 1$ be a short exact sequence of groups with excessive homology such that $K$ is finitely generated, $\Gamma^{ab}$ is infinite and $G$ is of type $FP_3$ with $\chi_3(G) > 0$.
Then $G$ is incoherent.
\end{theorem}

\begin{proof} By \cite[Theorem 1]{F-V} there is a normal subgroup $M$ of $G$ such that $G/ M \simeq \mathbb{Z}$ and $M$ is finitely generated. Suppose that $G$ is coherent. Then $M$ is finitely presented, in particular it is of type $FP_2$. Then for $\chi : G \to \mathbb{Z}$ with $M = Ker ( \chi)$ we have that $\pm [\chi] \in \Sigma^2(G, \mathbb{Z})$. Thus by \cite{S},  \cite[Appendix]{B} $[\chi] \in \Sigma^j(G, \mathbb{Z})$ if and only if  $Tor_i^{\mathbb{Z} G}( \mathbb{Z}, \widehat{\mathbb{Z} G}_{\chi}) = 0$ for $j \geq i$, here this applies for $i = 2$. Let $$\ldots \to \mathbb{Z} G^{k_3} \to \mathbb{Z} G^{k_2} \to \mathbb{Z} G^{k_1} \to \mathbb{Z} G^{k_0} \to \mathbb{Z} \to 0$$ be a free resolution of the trivial $\mathbb{Z} G$-module $\mathbb{Z}$ such that $0 < \chi_3(G) = k_0 - k_1 + k_2 - k_3$. Applying $- \otimes_{\mathbb{Z} G} \widehat{\mathbb{Z} G}_{\chi}$ we obtain an exact complex
\[\begin{tikzcd}
\widehat{\mathbb{Z} G}_{\chi}^{k_3} \arrow{r}{\partial_3}  & \widehat{\mathbb{Z} G}_{\chi}^{k_2}  \arrow{r}{\partial_2} & \widehat{\mathbb{Z} G}_{\chi}^{k_1}  \arrow{r}{\partial_1} & \widehat{\mathbb{Z} G}_{\chi}^{k_0}  \arrow{r}{\partial_0} & \mathbb{Z} \otimes_{\mathbb{Z} G} \widehat{\mathbb{Z} G}_{\chi} = 0  \arrow{r} & 0 \end{tikzcd}.
\]
Let $P_1 = Ker (\partial_1)$ and $P_2 = Ker (\partial_2)$.  Then since $\widehat{\mathbb{Z} G}_{\chi}^{k_0}$ is free as $\widehat{\mathbb{Z} G}_{\chi}$-module we have spliting
$\widehat{\mathbb{Z} G}_{\chi}^{k_1} \simeq  \widehat{\mathbb{Z} G}_{\chi}^{k_0} \oplus P_1$
and $P_1$ is a projective  $\widehat{\mathbb{Z} G}_{\chi}$-module. Hence
$ \widehat{\mathbb{Z} G}_{\chi}^{k_2} \simeq P_1 \oplus P_2$
and $P_2$ is a projective  $\widehat{\mathbb{Z} G}_{\chi}$-module. Finally there is an epimorphism $ \widehat{\mathbb{Z} G}_{\chi}^{k_3} \to P_2$ of $\widehat{\mathbb{Z} G}_{\chi}$-modules. This induces an epimorphism $$ \varphi :  \widehat{\mathbb{Z} G}_{\chi}^{k_1}  \oplus  \widehat{\mathbb{Z} G}_{\chi}^{k_3}  \to  \widehat{\mathbb{Z} G}_{\chi}^{k_2}  \oplus  \widehat{\mathbb{Z} G}_{\chi}^{k_0} $$ of $\widehat{\mathbb{Z} G}_{\chi}$-modules.
Note that since $M = Ker (\chi)$, $G/ M \simeq \mathbb{Z} = \langle t \rangle$, for $\chi_0 : G/ M \to \mathbb{Z}$ induced by $\chi$,  $\chi_0 (t) > 0$, we have that $\widehat{\mathbb{Z}[t^{ \pm 1}]}_{\chi_0} = \mathbb{Z}[[t]] [1/t] = R$ is a quotient of  $ \widehat{\mathbb{Z} G}_{\chi}$. Hence $\varphi$ induces an epimorphism of $R$-modules
$$\widetilde{\varphi} : R^{k_1 + k_3} \to R^{k_0 + k_2}$$
Since $R$ is a domain, this implies that $k_1 + k_3 \geq k_0 + k_2$, a contradiction.
\end{proof}

\begin{lemma}\label{Lemma CW}\cite{Wall}
Let $1\rightarrow C \rightarrow D \rightarrow E \rightarrow 1$ be a short exact sequence of groups of type $F$. Suppose that there are classifying spaces $K(C,1)$ and $K(E,1)$ with $\alpha_{t}(C)$ and $\alpha_{t}(E)$ $t$-cells, respectively. Then, there is a $K(D,1)$ complex with $\alpha_{i}(D)$ $i$-cells such that
\[ \alpha_{i}(D)= {\mathlarger{\sum}}_{0\leq t \leq i} \alpha_{t}(C)\alpha_{i-t}(E).\]
\end{lemma}

\begin{proof}[Proof of Corollary \ref{general}] For a group $G$ with a classifying space $K(G,1)$ with $\alpha_i(G)$ $i$-cells the Euler characteristics $\chi(G) = \sum_{i \geq 0} (-1)^i \alpha_i(G)$. Hence $\chi_3(G) = \chi(G) =  \chi(\Gamma) \chi(K) > 0$ and we can apply Theorem \ref{incoherent}. \end{proof}

\begin{proof}[Proof of Proposition \ref{Sigma2}] 
Observe  that each $G_i$ has a finite presentation with $a_i$ generators and $r_i$ relations, where $0 \leq r_i \leq 1$ and Euler characteristic $\chi(G_i) = 1 - a_i + r_i$. Then $\chi(G) = \chi(G_1) \chi(G_2) \not= 0$, hence $\chi(G_i) \not= 0$ for $i = 1,2$. On the other hand $ 1 - a_i + r_i \leq 2 - a_i \leq 0$ since $a_i = 1$ implies $G_i = \mathbb{Z}$ and $\chi(G_i) = 0$ a contradiction, hence $ \chi(G_i) = 1 - a_i + r_i < 0$ for each $i = 1,2$ and $\chi(G) = \chi(G_1) \chi(G_2) > 0$. 

 By Lemma \ref{Lemma CW} there is
a free resolution of the trivial $\mathbb{Z} G$-module $\mathbb{Z}$
$$0 \to \mathbb{Z} G^{k_4} \to  \mathbb{Z} G^{k_3} \to \mathbb{Z} G^{k_2} \to \mathbb{Z} G^{k_1} \to \mathbb{Z} G^{k_0} \to \mathbb{Z} \to 0$$ 
where $k_0 = 1, k_1 = a_1 + a_2, k_2 = r_1 + r_2 + a_1 a_2, k_3 = r_1 a_2 + r_2 a_1, k_4 = r_1 r_2$. 

Suppose that $[\chi] \in \Sigma^2(G, \mathbb{Z})$. Then by \cite{S}, \cite[Appendix]{B} $Tor_i^{\mathbb{Z} G}( \mathbb{Z}, \widehat{\mathbb{Z} G}_{\chi}) = 0$ for $1 \leq i \leq 2$. 
Then as in the proof of Theorem \ref{incoherent}  $k_1 + k_3 \geq k_0 + k_2$. This is eqivalent to
$$0 \leq k_1 + k_3 - k_0 - k_2 = a_1 + a_2 +  r_1 a_2 + r_2 a_1 - 1 - r_1 - r_2 - a_1 a_2 = r_1 r_2 - (1 - a_1 + r_1) ( 1 - a_2 + r_2) =$$ $$ r_1 r_2 - \chi(G_1) \chi(G_2) = r_1 r_2 - \chi(G) <  r_1 r_2 \leq 1.$$
Hence $0 =  r_1 r_2 - \chi(G)$, so $0 < \chi (G) = r_1 r_2 \leq 1$. This implies that $r_1 = r_2 = 1$ i.e. each $G_i$ is a surface group. Note that $0 > \chi(G_i) = 1 - a_i + r_i = 2 - a_i$ and $a_i$ even imply that $\chi(G_i)$ is even for $ i = 1,2$ . But $1 = r_1 r_2 = \chi(G) = \chi(G_1) \chi(G_2)$, a contradiction.

\end{proof}

\subsection{Some examples  $G = K \rtimes \Gamma$ with $K$ two-generated and $\Gamma$ free of finite rank $\geq 2$}

\begin{theorem} \label{excessive} Let $G = K \rtimes \Gamma$ be a short exact sequence of groups with $K$ finitely generated.
Suppose that there is  a normal subgroup $B$ of $K$ such that $A = K/ B \simeq \mathbb{Z}_2$ and for the abelianization $B^{ ab}$ of $B$ we have that $B^{ ab} \otimes_{\mathbb{Z}} \mathbb{Q} \simeq U \oplus \mathbb{Q}^t$ as left $\mathbb{Q} A$-module, where $A$ acts via conjugation, $U = Ker (\mathbb{Q} A \to \mathbb{Q})$ is the augmentation ideal and $\mathbb{Q}^t = \mathbb{Q} \oplus \ldots \oplus \mathbb{Q}$ with $t \geq 0$  direct summands and each $\mathbb{Q}$ a trivial $\mathbb{Q} A$-module. Then there is a subgroup $G_0$ of finite index in $G$ and a subgroup $\Gamma_0$ of $\Gamma$ such that $G_0 = (K \cap G_0) \rtimes \Gamma_0$ and the short exact sequence $ 1 \to K \cap G_0 \to G_0 \to \Gamma_0 \to 1$ has excessive homology.
\end{theorem}

\begin{proof} Since $A$ is a finite group and $K$ is finitely generated and normal in $G$, there is a subgroup $\Gamma_1$ of finite index in $\Gamma$ such that $g B g^{ -1}  = B$ for $g \in \Gamma_1$. This induces a $\Gamma_1$-action via conjugation on $A =K/ B$ and there is a subgroup $\Gamma_2$ of finite index in $\Gamma_1$ such that  $\Gamma_2$ acts trivially on $A$ i.e. $[K, \Gamma_2] \subseteq B$.

Let $a$ be a generator of $A$. By the isomorphism
 $ U \oplus \mathbb{Q}^t \simeq B^{ab} \otimes_{\mathbb{Z}} \mathbb{Q}$  of $\mathbb{Q} A$-module we have that  $U$ is isomorphic to the $\mathbb{Q}$-vector subspace of $ B^{ab} \otimes_{\mathbb{Z}} \mathbb{Q}$ generated by all $e$-vectors with $e$-value $-1$ with respect to the linear operator that is given by conjugation with $a$. And $\mathbb{Q}^t$
is isomorphic to the $\mathbb{Q}$-vector subspace of $ B^{ab} \otimes_{\mathbb{Z}} \mathbb{Q}$ generated by all $e$-vectors with $e$-value $1$ with respect to the linear operator that is given by conjugation with $a$.

 Set
$V_{1} = \{ v \in B^{ ab} / Tor \ | \ a \circ v = - v \} \hbox{ and } V_{2} = \{ v \in B^{ ab} /Tor \ | \ a \circ v =  v \},$
where $\circ$ denotes conjugation action and $Tor$ is the torsion part of the finitely generated abelian group $B^{ab}$. Then $V_1 \otimes_{\mathbb{Z}} \mathbb{Q} \simeq U$ is 1-dimensional over $\mathbb{Q}$ and $V = V_1 \oplus V_2$ is a subgroup of finite index in $B^{ab}/ Tor$.  Note that since $[K, \Gamma_2] \subseteq B$ we have that for every $g \in \Gamma_2$ the actions of $a$ and $g$ on $B^{ ab}$ ( induced by conjugation) commute.
 This implies that  $V_{1}$ and $V_{2}$ are $\Gamma_2$-invariant and $V_1 \simeq \mathbb{Z}$.
Then there is a subgroup $\Gamma_0$ of $\Gamma_2$ of index  2 such that $\Gamma_0$ acts trivially ( via conjugation) on $V_1$.
Finally we define $C$ as the preimage of $V_1 \oplus V_2$ in $B$ and $G_0 = C \rtimes \Gamma_0$. Thus $C$ has finite index in $K$ and the image of $V_1$ in the abelianization of $G_0$ is infinite cyclic i.e. $ 1 \to C \to G_0 \to \Gamma_0 \to 1$ has excessive homology.
\end{proof}

\begin{proposition}
Let $F$ be a free group with a free basis $a_1, \ldots, a_{m}$ and $B = Ker(\varphi)$, where $\varphi : F \to A = \mathbb{Z}_2^s$ is an epimorphism for some $1 \leq s \leq m$.
Then $B^{ ab} \otimes_{\mathbb{Z}} \mathbb{Q} \simeq  U^{m-1} \oplus \mathbb{Q}^m \simeq (\mathbb{Q} A)^{m-1} \oplus \mathbb{Q}$ as left $\mathbb{Q} A$-module via conjugation, where $U$ is the augmentation ideal $Ker(\mathbb{Q} A \to \mathbb{Q})$.
\end{proposition}

\begin{proof}
Note that $B/ B'$ is a relation module of $A$. Then by Brown's book \cite[Ch. 2, Prop. 5.4]{Brownbook} there is an exact complex of $\mathbb{Z} A$-modules
$$0 \to B/ B'\to (\mathbb{Z} A)^{ m} \to \mathbb{Z} A \to \mathbb{Z} \to 0$$
Then since $\otimes_{\mathbb{Z}} \mathbb{Q}$ is an exact functor we get an exact complex of $\mathbb{Q} A$-modules
$$0 \to B/ B'\otimes_{\mathbb{Z}} \mathbb{Q} \to (\mathbb{Q} A)^{m} \to \mathbb{Q} A \to \mathbb{Q} \to 0$$
Let $U = Ker (\mathbb{Q} A \to \mathbb{Q})$ be the augmentation ideal. Then we have 
an exact complex of $\mathbb{Q} A$-modules
$$0 \to B/ B'\otimes_{\mathbb{Z}} \mathbb{Q} \to (\mathbb{Q} A)^{ m} \to U \to 0$$
Since $U \oplus \mathbb{Q} = \mathbb{Q} A$, $U$ is a projective $\mathbb{Q} A$-module. Thus
\begin{equation} \label{eq1} U^m \oplus \mathbb{Q}^m \simeq (\mathbb{Q} A)^{m} \simeq U \oplus (B/ B'\otimes_{\mathbb{Z}} \mathbb{Q}) \end{equation}

Since $A$ has exponent 2 and the roots of $x^2 = 1$ are in $\mathbb{Q}$ the standard representation theory of $A$ over $\mathbb{C}$ actually works over $\mathbb{Q}$ i.e.  every finitely generated $\mathbb{Q} A$-module $V$ is a direct sum of irreducible $\mathbb{Q} A$-modules, each of dimension 1  over $\mathbb{Q}$ and the multiplicity of each irreducible factor in such direct sum decomposition of $V$ uniquely depends on $V$. This implies that if $V_1, V_2, V_3$ are finitely generated $\mathbb{Q} A$-modules with $V_1 \oplus V_2 \simeq V_1 \oplus V_3$ then $V_2 \simeq V_3$. Thus  canceling one factor $U$ in Equation (\ref{eq1}) we get that
$ U^{m-1} \oplus \mathbb{Q}^m  \simeq B/ B'\otimes_{\mathbb{Z}} \mathbb{Q}$.
\end{proof}

\begin{corollary} \label{new}
Let $G = K \rtimes \Gamma$ be a short exact sequence of groups with $K$ 2-generated such that $K$ has a non abelian quotient $T = \mathbb{Z} \rtimes \mathbb{Z}_2$. Then there is a subgroup $G_0$ of finite index in $G$ such that the short exact sequence $ 1 \to B = K \cap G_0 \to G_0 \to \Gamma_0 \to 1$ has excessive homology.
\end{corollary}

\begin{proof}
Let $K/ B_0 \simeq \mathbb{Z} \rtimes \mathbb{Z}_2$ be non-abelian. Then we have a sequence of subgroups
$$B'\subseteq B_0 \subseteq B \subseteq K$$
such that $A : = K/ B \simeq \mathbb{Z}_2$, $B / B_0 \simeq \mathbb{Z}$.  By the previous result $B/ B'\otimes \mathbb{Q}$ as a $\mathbb{Q} A$-module is a quotient of  $\mathbb{Q} A \oplus \mathbb{Q} = U \oplus \mathbb{Q}^2 $. By construction $B/ B'\otimes \mathbb{Q}$ maps surjectively to  $B/ B_0\otimes \mathbb{Q} = U = Ker (\mathbb{Q} A \to \mathbb{Q})$.
Hence  $B/ B'\otimes \mathbb{Q} \simeq U \oplus \mathbb{Q}^t$ for some $ 0 \leq t \leq 2$.
Then we can apply Theorem \ref{excessive}.
\end{proof}

\begin{proof}[Proof of Corollary \ref{gen}] It follows directly from  Corollary \ref{new} and the homological version of Corollary \ref{cor1}. \end{proof}

\begin{examples} The case $K$ free of rank 2 is one obvious case when the conditions of Corollary \ref{gen} hold.
 There are other examples of Corollary \ref{gen} like $K_1 = \langle x,y \ | \ (xy)^2 = 1 \rangle$ and $K_2 = \langle x,y \ | \ xyx^2 y x y^2 = 1 \rangle$ that satisfy Corollary \ref{gen} and are not free.  The second example is torsion-free and the first has torsion. Note that $K_1 = \langle x, z \ | \ z^2 = 1 \rangle = \langle x, \ ^zx \rangle \rtimes \langle z \ | \ z^2 = 1 \rangle$, where $T = \langle x, \ ^zx \rangle $ is a free group of rank 2. If $H$ is a normal subgroup of $K_1$ that is finitely generated and  $K_1/ H \simeq \mathbb{Z}$ then $T \cap H$ is a finitely generated, normal subgroup of infinite index in the free group $T$,  a contradiction since $T$ does not algebraically fiber. 

Alternatively the non-fibering part for a 1-relator group can be checked by the Brown's criterion \cite[Theorem 4.2]{Brown2}.  We use this approach for $K_2$. Indeed assume that  $H$ is a normal subgroup of $K_2$ that is finitely generated and  $K_2/ H \simeq \mathbb{Z}$. Thus $H = Ker (\chi)$, where $\chi : K_2 \to \mathbb{Z}$ with $\chi(x) = a, \chi(y) = b$. Then $0 =\chi(xyx^2 y x y^2 ) = 4a + 4b$, so $b = -a \not= 0$. By the Brown's criterion in the sequence $(\chi(s) \ | \ s = x_1 \ldots x_i\hbox{ for }0 \leq i \leq 7, x_1 \ldots x_8 = xyx^2 y x y^2) $ both the minimal and maximum are attained exactly once. Note that  the sequence $(\chi(s))$ is $0,a, a+ b = 0, 2a + b = a, 3a + b = 2a, 3a + 2b = a, 4a + 2b = 2a, 4a + 3b = a$ and both the minimum and the maximum are attained more than once. \end{examples} 

\section{Applications for poly-free and poly-surface groups} \label{pfs}

In this section we will apply Theorem \ref{thm:poly1}  to study higher finiteness properties of algebraic fibrations of  poly-free and poly-surface groups. These classes of groups, that have been extensively studied both in algebraic and geometric setting, generalize the class of surface-by-surface groups, that was the main motivation of \cite{F-V}. In what follows, the term {\em surface group of genus $g$} will refer to the fundamental group of a connected, compact, closed, orientable surface of genus $g > 0$. Nonabelian free and surface groups share an important property, the so-called f.g.n. (finitely generated normal) property, namely all their non-trivial finitely generated normal subgroups are finite-index, see \cite{KS73}. As a consequence, their first BNS invariant is empty.

\begin{definition} A {\em poly-free} (respectively {\em poly-surface}) {\em group} $G$ of length $m$ is a group that admits a subnormal filtration $1  = G_0 \unlhd G_1 \unlhd \ldots \unlhd G_m = G$ such that for each $i = 1,\ldots,m$ the {\rm factor group} $\Gamma_i := G_{i}/G_{i-1}$  is a finitely-generated free (respectively surface) group. In particular it admits a short exact sequence of type $1 \to K \to G \to \Gamma \to 1$ with $K$ a poly-free (respectively poly-surface) group of length $m-1$ and $\Gamma$ a free (respectively surface) group.
\end{definition} 

Direct products of free (respectively surface) groups belong to these classes. The fundamental group of a surface bundle over a surface is a poly-surface group of length $2$ (and nontrivially the converse holds true), and iterations of such geometric construction afford examples of higher-length poly-surface groups. In the poly-free case the sequence $1 \to G_{i-1} \to G_{i} \to G_{i}/G_{i-1} \to 1$ always splits, but it may fail to do so in the poly-surface case. Whenever the sequence splits, we can write $G_{i}$ as semidirect product $G_{i} = G_{i-1} \rtimes_{\rho} (G_{i}/G_{i-1})$ determined by a representation $\rho \colon G_{i}/G_{i-1} \to \operatorname{Aut}(G_{i-1})$. 

By their very definition, poly-free and poly-surface groups admit a subnormal filtration whose factor groups are finitely presented $FP$ groups, hence by finite induction they are themselves finitely presented and $FP$ (see e.g. \cite[Section VIII.6]{Brownbook}), hence $F_{\infty}$. Furthermore, in the poly-surface case, as the factor groups are $PD(2)$ groups, they are $PD(2m)$ groups, see \cite[Theorem 9.10]{Bieribook}.  As $FP$ groups the Euler characteristic of  a poly-free or poly-surface group $G$
 \[ \chi(G) = \sum_{p} (-1)^{p} \mbox{rk}_{\Z}(H_p(G);\Z) \] is well-defined (\cite[Section IX.6]{Brownbook}). It can be computed inductively from its subnormal filtration, using multiplicativity  
under $FP$-extensions  (\cite[Proposition 7.3(d)]{Brownbook}) to get \[ \chi(G) = \prod_{i=1}^{m} \chi(G_{i}/G_{i-1}). \] Observe that $\chi(G)$ is zero if and only if one of the factor groups is abelian (i.e. $\Z$ or $\Z^2$ for the free and surface group respectively). 

 The examples of poly-free and poly-surface groups we will consider admit a normal filtration. (Such groups have been often referred to as {\em strongly} poly-free and poly-surface.) Note that such condition is not exceptional: in fact, whenever the Euler characteristic of a poly-free or poly-surface group $G$ is nonzero, one may assume that there is a finite index subgroup of $G$ whose filtration is actually normal, see \cite[Corollary 4]{Me84} and \cite[Theorem 9]{J79}.

 Before stating the main results of this section, we want to remark that there is a characterization of a poly-free (or poly-surface) group admitting a normal filtration, as iterated extensions of its factor groups. This observation will facilitate the application of Theorem \ref{thm:poly1} for normal filtrations.

\begin{remark} \label{rem}
1) Given a poly-free or poly-surface with normal filtration $1  = G_0 \unlhd G_1 \unlhd \ldots \unlhd G_m = G$, it is convenient to define the groups $\Pi(m-j) := G_{m}/G_j, \,\ 0 \leq j \leq m$. In particular, $\Pi(m) = G$ and $\Pi(0) = 1$; it is obvious from this definition that $G_j= Ker(\Pi(m) \to \Pi(m-j))$. Furthermore, $\Pi(m-j)$ admits a normal filtration $1  = H_0 \unlhd H_1 \unlhd \ldots \unlhd H_{m-j} = \Pi(m-j)$, where the $H_i := Ker(\Pi(m-j) \to \Pi(m-j-i))$ satisfies the relation $H_{i} = G_{j+i}/G_{j}, \,\ 0 \leq  i \leq m-j$, as follows from the commutative diagram
\[ \xymatrix{
G_j \ar[d]_{\cong} \ar@{^{(}->}[r] &
G_{j+i} \ar@{^{(}->}[d] \ar@{->>}[r]  & H_i \ar@{^{(}->}[d] \\
G_j \ar@{^{(}->}[r] &
\Pi(m) \ar@{->>}[d]\ar@{->>}[r] & \Pi(m-j) \ar@{->>}[d] \\
 & \Pi(m-j-i) \ar[r]^{\cong}   &
\Pi(m-j-i)  } \]
so the $i$-th factor $H_{i}/H_{i-1}$ equal to $\Gamma_{j+i} =G_{j+i}/G_{j+i-1}$. It follows that $\Pi(m-j)$ is itself a poly-free or poly-surface group of length $m-j$, and each of these groups is related with its predecessor via the sequence \begin{equation} \label{eq:iter} 1 \to \Gamma_{j+1} \to \Pi(m-j) \to \Pi(m-j-1) \to 1 \end{equation} for $0 \leq j \leq m-1$, as $\Gamma_{j+1}  = G_{j+1}/G_{j}$. Namely, $\Pi(m)$ can be defined inductively from the trivial group $\Pi(0)$ via extensions whose fibers are the factor groups of its normal series. 

2) Conversely, assume we are given a a collection of groups $\Pi(m-j), \,\ 0 \leq  j \leq m$, with $\Pi(0) = 1$ and the remaining groups $\Pi(m-j), \,\ m-1 \geq j \geq 0$ defined inductively by a sequence as in Equation (\ref{eq:iter}),  where the groups $\Gamma_{j+1}$ are f.g. free (or surface) groups; then each group $\Pi(m-j)$ is a poly-free or poly-surface group of length $(m-j)$ with normal filtration given by $H_i = Ker(\Pi(m-j) \to \Pi(m-j-i))$ for $0 \leq i \leq m-j$ and $i$-th factor equal to $\Gamma_{j+i}$, as determined by the commutative diagram
\[ \xymatrix{
H_{i-1} \ar[d]_{\cong} \ar@{^{(}->}[r] &
H_i \ar@{^{(}->}[d] \ar@{->>}[r] & \Gamma_{j+i} \ar@{^{(}->}[d] \\
H_{i-1} \ar@{^{(}->}[r] &
\Pi(m-j) \ar@{->>}[d]\ar@{->>}[r] & \Pi(m-j-i+1) \ar@{->>}[d] \\
 & \Pi(m-j-i) \ar[r]^{\cong}   &
\Pi(m-j-i) } \]
Denoting as above $1  = G_0 \unlhd G_1 \unlhd \ldots \unlhd G_m = \Pi(m)$ the normal filtration for the case $q=0$, this recasts the isomorphism $\Pi(m-i) \cong G_{m}/G_i, \,\ 0\leq i \leq m$, so the notation is consistent.  
\item The advantage of this reformulation is that the assumption  a') of Theorem \ref{thm:poly1} can be rephrased in terms of excessive homology of the sequences $1 \to \Gamma_{j+1} \to \Pi(m-j) \to \Pi(m-j-1) \to 1$ for $0 \leq j \leq m-1$. In particular, if we describe $\Pi(m)$ inductively, we need to  to verify at each step the existence of excessive homology only for the single sequence defining the group $\Pi(m-j)$.  
\end{remark}

\begin{proof}[Proof of Theorem \ref{thm:poly}]
The assumption that  $H_1(G;\Z) \cong \oplus_{i=0}^{m-1}H_1(G_{i+1}/G_i;\Z)$ 
 entails (and is in fact equivalent to) the fact that for all $0 \leq j \leq m-1$ the short exact sequence $1 \to G_{j} \to G_{j+1} \to G_{j+1}/G_{j} \to 1$ 
induces an isomorphism $H_{1}(G_{j+1};\Z) \cong H_{1}(G_{j};\Z) \oplus H_{1}(G_{j+1}/G_{j};\Z)$; in particular this 
implies that every discrete character $G_j \to \mathbb{R}$ extends to a discrete character   $G_{j+1} \to \mathbb{R}$ for every $ 1 \leq j \leq m-1$, and that there is a discrete character $G_{j+1} \to \mathbb{R}$ whose restriction to $G_j$ is zero, and that can be further  extended to a discrete character $\widetilde{\alpha}_j : G \to \mathbb{R}$. Then we can apply Theorem \ref{thm:poly1}  for $s = 0$ and an arbitrary discrete character $\phi : G_1 \to \mathbb{R}$ to obtain a discrete character $\psi : G \to \mathbb{R}$ that extends $\phi$ and  $M = Ker(\psi) $ is of type $F_{m-1}$. Note that when the filtration is normal, the assumption $H_1(G;\Z) \cong \oplus_{i=0}^{m-1}H_1(G_{i+1}/G_i;\Z)$ is equivalent also to the fact that for all $0 \leq j \leq m-1$, in the short exact sequence \[ 1 \to \Gamma_{j+1} \to \Pi(m-j) \to \Pi(m-j-1) \to 1, \] the action of the base groups $\Pi(m-j-1)$ on the homology of the fibers $H_{1}(\Gamma_{j+1};\Z)$ is trivial; namely the monodromy representations $\Pi(m-j-1) \to Out(\Gamma_{j+1})$ have values in the corresponding Torelli subgroup. 

 In order to show that, whenever $\chi(G) \neq 0$, $M$ is not $FP_m$, we will use an Euler characteristic argument, with slight differences in the free and the surface cases. For the poly-free case, the very definition of the group $G$ shows its cohomological dimension equal its length $\operatorname{cd}(G) = m$, as it is an iterated extension of $FP$ groups of cohomological dimension $1$ (see \cite[Remarks, pg. 72]{Bieribook}). It follows that the subgroup $M \leq G$ must satisfy $\operatorname{cd}(M) \leq \operatorname{cd}(G) = m$ (see \cite[Proposition 5.9]{Bieribook}. If  $M$ were $FP_m$ it would be $FP$, hence each term of the short exact sequence \begin{equation} \label{algfib} 1 \longrightarrow M \longrightarrow G \stackrel{\psi}{\longrightarrow} \Z \longrightarrow 1 \end{equation} would have a well-defined Euler characteristic.  Using multiplicativity of the Euler characteristic we would thus get 
$\chi(G) = \chi(M) \cdot \chi(\Z) = 0$. When $G$ is a poly-surface group of length $m$, it has cohomological dimension $2m$, this argument would just yield the weaker result that $M$ is not $FP_{2m-1}$; fortunately there is an off-the-shelf stronger argument that can be applied using the fact that $G$ is a Poincar\'e duality group of dimension $2m$ i.e. a $PD(2m)$ group. For in that case, in the sequence of Equation \ref{algfib},  $M$ is a $PD(2m-1)$ group if and only if it is of type $FP_{m}$  (see \cite[Theorem 1.19]{Hibook}). So if we assume the latter, we have much as above that the Euler characteristic of $M$ is well-defined and we must have $\chi(G) = \chi(M) \cdot \chi(\Z) = 0$. \end{proof}

\begin{remark} 
(1) All groups of nonzero Euler characteristic covered by  Theorem  \ref{thm:poly} fail to be  $n$-coherent and homologically $n$-coherent for all $n \leq m - 1$. For $n < m - 1$, the incoherence is witnessed by subgroups of elements of the subnormal series, that may fail to be normal in $G$ itself.   

(2) Note that in the case where $G$ is the direct product of $m$ surface groups, Theorem \ref{Main1} can be seen as a consequence of Meinert's inequality for BNSR invariants, and Theorem \ref{thm:poly} can be obtained as a particular case of the much broader results on finiteness of subgroups of direct products contained in \cite{BHMS}.
\end{remark}

 We want to apply Theorem \ref{thm:poly} to the study of the finiteness properties of subgroups of two poly-free groups of geometric interest, namely the pure braid group and the pure mapping class group of a punctured sphere. We will primarily discuss normal subgroups with associated quotient groups that  equal to $\Z$, i.e. kernels of discrete characters: their study gives some information on the BNSR invariants of the group. The BNSR invariants of the pure braid group have received already attention: in particular the first invariant is completely determined in \cite{KMcCM}, while (portions of) the higher order invariants has been partially elucidated in \cite{Z2}. Corollary \ref{cor:pbg} recasts results similar to those of \cite{Z2}, using techniques that at times are quite different (at least on the surface); in particular, no use is made of Morse theory.

\begin{proof}[Proof of Corollary \ref{cor:pbg}] We start recalling some basic algebraic results on the groups involved, which can be found e.g. in \cite[Section 9.3]{FM}. The pure braid group on two strands $P(2)$ is isomorphic to $\Z$ and cancellation of one strand from a braid yields a short exact sequence $1 \to F_m \to P(m+1) \to P(m) \to 1$.  Iteration of the cancellation process yields the normal filtration  $1  = G_0 \unlhd G_1 \unlhd \ldots \unlhd G_m = P(m+1)$ of length $m$, where $G_{i} = Ker(P(m+1) \to P(m+1-i)) $ is the kernel of the map obtained by successive removal of $i$ strands, so that the factor groups $G_{i}/G_{i-1}$ are free of rank $m+1-i$. Truncation of the filtration at the penultimate step yields that $P(m+1)$ is an extension \begin{equation} \label{eq:pure} 1 \longrightarrow K(m+1) \longrightarrow P(m+1) \longrightarrow \Z  \longrightarrow 1, \end{equation} where $K(m+1) = G_{m-1}$ is a poly-free group of length $m-1$, that can be identified with the pure mapping class group of a sphere with $m+2$ punctures.  Note that in the case of $K(m+1)$ all factors are free of rank at least $2$, hence its Euler characteristic is nonzero, while the extra factor of $P(m+1)$ is $\Z$, so that its Euler characteristic vanishes. Furthermore, we have an actual direct product decomposition $P(m+1) = K(m+1) \times \Z$, with $\Z = Z(P(m+1))$ being the center of $P(m+1)$ (\cite[Section 9.3]{FM}).

Both groups $P(m+1)$ and $K(m+1)$, with their normal filtrations, fit in the framework described in Remark \ref{rem}, and can be thought of as the result of successive extensions related by the commutative diagrams
\begin{equation} \label{eq:diag} \xymatrix{
F_{m-j} \ar[d]_{\cong} \ar@{^{(}->}[r] &
K(m+1-j) \ar@{^{(}->}[d] \ar@{->>}[r] & K(m-j) \ar@{^{(}->}[d] \\
F_{m-j} \ar@{^{(}->}[r] &
P(m+1-j) \ar@{->>}[d]\ar@{->>}[r] & P(m-j) \ar@{->>}[d] \\
 & \Z \ar[r]^{\cong}   &
\Z } 
\end{equation}
 for $m-2 \geq j \geq 1$. (Rather unfortunately, the indices for the number of strands, punctures, and poly-free length are all offset.)

The first homology group of $P(m+1)$ is $\Z^{{m+1}  \choose 2}$. A quick check verifies that its rank is exactly the sum of the ranks of all factors of the normal series. Therefore the condition $H_1(G;\Z) \cong \oplus_{i=0}^{m-1}H_1(G_{i+1}/G_i;\Z)$ is satisfied, which entails that the assumption of Theorem \ref{thm:poly} are satisfied for $K(m+1)$ as well. 

We can then proceed with the rest of the proof, looking at $K(m+1)$ first, by induction. Observe that for $m=2$ we have $K(3) = F_2$ and all infinite index normal subgroups of $F_2$ are of type $F_0$ but not $FP_1$. Assume that the statement holds true for $K(m)$, and we will prove this for $K(m+1)$. By inductive assumption $K(m)$ contains for all $0 \leq n \leq m-3$ a discrete character $\phi_n \colon K(m) \to \R$ whose kernel $Ker(\phi_n) \unlhd K(m)$ is of type $F_{n}$ but not $FP_{n+1}$. The discrete character $\phi_n \circ \pi \colon K(m+1) \to \R$ (where $\pi \colon K(m+1) \to K(m)$) has kernel that fits in the sequence $1 \to F_m \to  Ker (\phi_n \circ \pi) \to Ker( \phi_n) \to 1$ hence, as $F_m$ is of type $F_{\infty}$, it has the same homotopical and homological finiteness properties as $Ker(\phi_n)$. To complete the statement, we need an algebraic fiber of $K(m+1)$ of type $F_{m-2}$ but not of type $FP_{m-1}$: but this is afforded by the fact that $K(m+1)$ is a poly-free group of length $m-1$ satisfying the assumptions of Theorem \ref{thm:poly} and nonzero Euler characteristic. (Note that this does not completely determine the finiteness properties of the algebraic fibers of $K(m+1)$.)  This completes the proof for $K(m+1)$. 

For the case of $P(m+1) = K(m+1) \times \Z$, each discrete character $\psi \colon K(m+1) \to \R$ determines a discrete character $\xi \colon P(m+1) \to \R$ that evaluates trivially on the $\Z$-factor. The kernel of $\xi$ equals $Ker(\psi) \times \Z$ hence will have the same regularity as $Ker(\psi)$, so that $P(m+1)$ admits for all $0 \leq n \leq m-2$ a discrete character with kernel of type $F_{n}$ but not $FP_{n+1}$. There is however one significant distinction from the case of $K(m+1)$: any discrete character of $P(m+1)$ that evaluates nontrivially on the $\Z$-factor (which is the center of $P(m+1)$) has kernel that is automatically of type $F_{\infty}$: this is not obstructed by the second part of Theorem \ref{thm:poly} because the Euler characteristic of $P(m+1)$ vanishes. Note that, here too, this does not completely determine the finiteness properties of the algebraic fibers of $P(m+1)$.
\end{proof}

\begin{remark} The reader may have noticed that the only properties of the pure braid group that are used are the fact that their monodromy representation induced by the sequences $1 \to \Gamma_{j+1} \to G_{m}/G_{j} \to G_{m}/G_{j+1} \to 1$ for $0 \leq j \leq m-1$ is Torelli, namely the entire homology group $H_1(\Gamma_{j+1};\Z)$ is coinvariant, and the isomorphism $G_{m} = G_{m-1} \times Z(G_{m}) = G_{m-1} \times \Z$. It follows that the conclusions of Corollary \ref{cor:pbg} apply verbatim for any group whose normal series satisfies those properties. An example of that kind is the upper McCool group  $P\Sigma^{+}_{m+1}$, a suitable subgroup of $Aut(F_{m+1})$, similar but distinct from $P(m+1)$; the properties required here can be found in \cite{CP,SW}. \end{remark}

Next, we will apply Theorem \ref{thm:poly1} to study finiteness properties of algebraic fibers of poly-surface groups that arise as fundamental groups of complex projective varieties.

\begin{proof}[Proof of Corollary \ref{cor:itko}] The variety $X(m)$ can be taken to be the result of iterations of the Atiyah-Kodaira construction of poly-surface groups of length $2$ with excessive homology. Technically, the induction could start by taking $X(1)$ to be a Riemann surface of genus at least $2$, but for sake of presentation we will give the details of the construction of $X(2)$, the classical Atiyah--Kodaira fibration. Given a Riemann surface of genus at least $2$ we can pass to a regular (unramified) cover of degree two to obtain a Riemann surface $R$ (of genus $g(R)  \geq 3$) admitting a fixed point-free involution $\iota \colon R \to R$. Denote by $f \colon S \to R$  the regular (unramified) cover determined by the (choice of an) epimorphism $\pi_{1}(R) \twoheadrightarrow H_{1}(R;\Z_{2}) \cong \Z_{2}^{2g(R)}$. The union of the graphs $\Gamma_{f} \cup \Gamma_{\iota \circ f} \subset S \times R$ of the maps $f, \iota \circ f \colon S \to R$ is a divisor whose homology class has divisibility $2$.  The double cover $p \colon X \to S \times R$  ramified along $\Gamma_{f} \coprod \Gamma_{\iota \circ f}$ is therefore well--defined, and the composition of $p$ with the projection maps of $S \times R$ to its factors yields two fibrations $\pi_1 \colon X \to S$, $\pi_{2} \colon X \to R$ that have fibers that are ramified covers of $R$ and $S$ respectively, in particular $X$ is doubly--fibered. One can determine (not without some effort) various topological invariants of $X$ in terms of the genus of $R$, but the only two facts of relevance here are that the Euler characteristic of $X$ is nonzero, and that $b_1(X) > \operatorname{max}\{b_1(S),b_1(R)\}$. The latter is a consequence of the fact that any double cover with nonempty ramification locus induces an epimorphism at level of fundamental groups. Denoting by $F$ the double cover of $R$ ramified in $2$ points, which is the fiber of the fibration  $\pi_1 \colon X \to S$, the sequence 
\begin{equation} \label{eq:ak} 1 \longrightarrow \pi_1(F) \longrightarrow \pi_1(X) \longrightarrow \pi_1(S) \longrightarrow 1 \end{equation} has excessive homology. Suitable finite index  (f.i. henceforth) subgroups of $\pi_1(X)$ and $\pi_1(S)$ will correspond, in the notation of Remark \ref{rem}, to the groups $\Pi(2)$ and $\Pi(1)$, and it is known that they admit algebraic fibers with the propertis desired (see \cite{FV21} for the case of $\Pi(2)$). 

In what follows, we will frequently pass to finite index subgroups of a given group, and to avoid to overburden the text, we will maintain when convenient the original notation; note that the intersection of a (sub)normal filtration of a group with one of its finite index subgroups yields a (sub)normal filtration of the latter. Furthermore, we will make repeated use of the following fact: given a short exact sequence $1 \to K \to G \to \Gamma \to 1$ with excessive homology, and a subgroup $G_0 \leq_{f.i.} G$, the induced sequence $1 \to K_0 \to G_0 \to \Gamma_0 \to 1$ will have excessive homology as well. In fact, as the rational homology of a finite index subgroups surjects on that of the group,  the inverse image of any element of $H_1(K;\Q)$ with nontrivial image in $H_{1}(G;\Q)$ will have nonzero image in $H_{1}(G_0;\Q)$.

The Atiyah--Kodaira construction above is the step $m=2$ of the induction process. The inductive step was developed for different purposes and should be credited to \cite{LIP} (see also \cite{Mi}), where it is described in detail; we will limit ourselves to summarizing it in order to prove that the varieties therein defined satisfy the properties relevant to us, referring the reader to \cite[Section 5]{LIP} for justification of the construction.

Assume then that for $0 \leq j \leq m$ we have a collection of complex projective variety $X(m-j)$ with $\operatorname{dim}_{\C}(X(m-j)) = m-j$ with $X(0) = \ast$ such that there exists a smooth fibration (i.e. a holomorphic submersion) $X(m-j)  \to X(m-j -1)$ with fiber an oriented surface of genus at least $2$. Their fundamental groups $\Pi(m-j) := \pi_1(X(m-j))$ are poly-surface groups of length $(m-j)$ and nonzero Euler characteristic related by short exact sequences \[ 1 \to \Gamma_{j+1} \to \Pi(m-j) \to \Pi(m-j-1) \to 1 \] for $0 \leq j \leq m-1$, where the $\Gamma_{i}$ for $1 \leq i \leq m-j$ are factor groups of a normal filtration of $\Pi(m-j)$. Assume furthermore that all these sequences have excessive homology and that for all $0 \leq n \leq m-j-1$, $\Pi(m-j)$ admits a cocylic subgroup of type $F_n$ but not $FP_{n+1}$.

We want to show that there exists a complex projective varieties $X(m+1)$ admitting a smooth fibration $X(m+1) \to X(m)$ with fiber an oriented surface of genus at least $2$ whose fundamental group $\Pi(m+1)$ is a a poly-surface group of length $(m+1)$ given as an extension \[ 1 \to \Gamma_0 \to \Pi(m+1) \to \Pi(m) \to 1 \] with fiber $\Gamma_0$ a surface group, so that the extension has excessive homology. (Note that the ``addition" of one group to a collection as in Remark \ref{rem} entails an extension of each term of the existing normal filtration, hence one new factor group - the extension of the trivial group! - labeled here as $\Gamma_0$, that add to the existing collection of factor groups $\Gamma_{i}$ for $1 \leq i \leq m$.)  Start with the fibration $F \hookrightarrow X(m) \rightarrow X(m-1)$ with associated sequence $1 \to \pi_1(F) \to \pi_1(X(m)) \to \pi_1(X(m-1)) \to 1$ where we $\pi_1(F) = \Gamma_ 1$.
Choose any epimorphism $\pi_1(F) \twoheadrightarrow H_1(F;\Z_{2}) \twoheadrightarrow \Z_2$. Up to going to a f.i. subgroup of $\pi_{1}(X(m-1))$ (and the corresponding f.i. subgroup of $\pi_1(X(m))$) if necessary we can assume (\cite[Proposition 37]{LIP}) the epimorphism $\pi_1(F) \twoheadrightarrow \Z_2$ extends to an epimorphism $\pi_1(X(m)) \twoheadrightarrow \Z_2$ whose associated unramified cover we will denote as $X'(m)$. Denote by $R \hookrightarrow X'(m) \rightarrow X'(m-1)$ the resulting fibration (where $R$ is double cover of $F$ and actually $X'(m-1) = X(m-1)$). By construction $X'(m)$ admits a fixed point-free involution $\iota  \colon X'(m) \to X'(m)$ (given by the $\Z_2$ deck transformation) which restricts fiberwise to an involution of the fiber $R$.  Next, choose an epimorphism $\pi_{1}(R) \twoheadrightarrow H_{1}(R;\Z_{2}) \cong \Z_{2}^{2g(R)}$. Again, up to going to a f.i. subgroup of $\pi_{1}(X'(m-1))$ if necessary we can assume (\cite[Proposition 37]{LIP}) the epimorphism $\pi_1(R) \twoheadrightarrow \Z_2^{2g(R)}$ extends to an epimorphism $\pi_1(X'(m)) \twoheadrightarrow \Z_{2}^{2g(R)}$ whose associated unramified cover we will denote as $X''(m)$. Denote by $S$ the fiber of the induced fibration of $X''(m)$; the base of this fibration can be assumed to be $X'(m-1)$, and the covering map $f \colon X"(m) \to X'(m)$ restricts fiberwise to a covering map $S \to R$ with the same degree. We can now consider the fiber product $X'(m) \times_{X'(m-1)} X''(m)$ (this can also be thought of as pull-back to $X''(m)$ of the fibration $X'(m) \to X'(m-1)$, or viceversa); importantly, it is a complex projective variety that fits in the commutative diagram
\[ \xymatrix@=12pt{ 
  S \times R  \ar@{->>}[d]\ar@{->>}[r] \ar@{^{(}->}@<-3pt>[dr] \hole & \hspace{3pt} S \hole \ar@{^{(}->}@<-2pt>[d] \ar[r]^{\cong} &  \hspace{3pt} S \hole \ar@{^{(}->}@<-2pt>[d]   
\\  R \hole \ar@{^{(}->}@<-1pt>[r] \ar[d]_{\cong} & X'(m) \times_{X'(m-1)} X''(m) \ar@{->>}[d] \ar@{->>}[r] \ar@{->>}@<-3pt>[dr] \hole & X''(m) \ar@{->>}[d]     
\\   R \hole \ar@{^{(}->}@<-1pt>[r] & \hspace{3pt} X'(m) \hole \ar@{->>}[r] & X'(m-1)     } \] 
where all surjective maps are holomorphic fibrations. Similarly, the fundamental groups of the varieties involved are related in the following commutative diagram of short exact sequences, that exhibits the fundamental group of $X'(m) \times_{X'(m-1)} X''(m)$ as fiber product $\pi_1(X'(m)) \times_{\pi_{1}(X(m-1))} \pi_1(X''(m))$ (this affords a direct proof, or can be seen as a consequence of the Mayer-Vietoris fiber sequence) :
\begin{equation} \label{dia:9} \xymatrix@=12pt{ 
  \pi_1(S \times R)  \ar@{->>}[d]\ar@{->>}[r] \ar@{^{(}->}@<-3pt>[dr] \hole & \hspace{3pt} \pi_1(S) \hole \ar@{^{(}->}@<-2pt>[d] \ar[r]^{\cong} &  \hspace{3pt} \pi_1(S) \hole \ar@{^{(}->}@<-2pt>[d]   
\\  \pi_1(R) \hole \ar@{^{(}->}@<-1pt>[r] \ar[d]_{\cong} & \pi_1(X'(m) \times_{X'(m-1)} X''(m))  \ar@{->>}[d] \ar@{->>}[r] \ar@{->>}@<-3pt>[dr] \hole & \pi_1(X''(m)) \ar@{->>}[d]     
\\   \pi_1(R) \hole \ar@{^{(}->}@<-1pt>[r] & \hspace{3pt} \pi_1(X'(m)) \hole \ar@{->>}[r] & \pi_1(X'(m-1))     } \end{equation}
We will (at least implicitly) prove that $X'(m) \times_{X'(m-1)} X''(m)$ satisfy the conclusions of the Theorem, but it has the blemish that its fundamental group would just be a ``trivially extended" version of that of $X(2)$ (as witnessed by the fact that we can endow it with a normal series whose  term $G_2$ is the product $\pi_1(S) \times \pi_1(R)$, and each term of the normal series up the penultimate term $G_{m}$ is a trivial extension (a direct product) of the terms of the normal series for $\pi_1(X''(m))$. In order to avoid this nuisance, we want to perform the Atiyah--Kodaira construction to the fibers of the fibration $X'(m) \times_{X'(m-1)} X''(m) \to X'(m-1)$. In order to do so we proceed as follows. 
Recall that $X'(m)$ admits a fixed point-free involution inducing an involution on the fiber $R$; also, recall that the covering map $f \colon X''(m) \to X'(m)$ restricts fiberwise to a covering map $S \to R$.  The image $\Gamma_{f} \cup \Gamma_{\iota \circ f} \subset X''(m) \to X''(m) \times_{X'(m-1)} X'(m)$ of the sections $(f,1), (\iota \circ f,1) \colon  X''(m) \to X'(m) \times_{X'(m-1)} X''(m)$ is a divisor that restricts, on $S \times R$, to the divisor used in the Atiyah--Kodaira construction. However, as observed in \cite{LIP}, it is not obvious that the construction can be done ``in family" (i.e. over $X'(m-1)$): this may require to  pass once again to a finite cover of $X'(m-1)$ (see \cite[Proposition 41]{LIP}).  Once this step is accomplished, we can take the double cover of (the finite cover of) $X'(m) \times_{X'(m-1)} X''(m)$ ramified over the  (finite cover of) $\Gamma_{f} \cup \Gamma_{\iota \circ f}$ to obtain an complex projective variety $X(m+1)$ with $\operatorname{dim}_{\C}(X(m+1)) = m+1$. It admits a holomorphic submersion $\pi \colon X(m+1) \to X''(m)$ with fiber $F$ which is a double cover of $R$ ramified in $2$ points. Its fundamental group  sits in the commutative diagram, which inherits the second and third line form that in Equation \ref{dia:9}, 
\[ \xymatrix@=12pt{ 
  \pi_1(F)  \ar@{->>}[d]\ar@{^{(}->}[r] \hole & \hspace{3pt} \pi_1(X(m+1)) \hole \ar@{->>}@<-2pt>[d] \ar@{->>}[r] & \hspace{3pt} \pi_1(X''(m)) \hole \ar[d]^{\cong}   
\\  \pi_1(R) \hole \ar@{^{(}->}@<-1pt>[r] \ar[d]_{\cong} & \pi_1(X'(m) \times_{X'(m-1)} X''(m))  \ar@{->>}[d] \ar@{->>}[r]  & \pi_1(X''(m)) \ar@{->>}[d]     
\\   \pi_1(R) \hole \ar@{^{(}->}@<-1pt>[r] & \hspace{3pt} \pi_1(X'(m)) \hole \ar@{->>}[r] & \pi_1(X'(m-1))     } \] 
(Note that only the horizontal sequences are exact.) The vertical maps between the first and second row are epimorphisms induced by the double ramified covers.  As $\pi_1(X''(m))$ has nonzero Euler characteristic, so will $\pi_1(X(m+1))$. The sequence in the third row has excessive homology, as it is (up to passage to finite index subgroups) one of the sequences which do by inductive assumption, hence so does the sequence in the first. 

In summary, after relabeling all the finite covers of the various $X(m-j)$ for $0 \leq j \leq m$, we end up with the candidate family of complex projective varieties, and all that is left is to prove that $\Pi(m+1) := \pi_1(X(m+1))$ has algebraic fibers with the stated properties. (The $\Pi(m-j)$, even if they are finite index subgroups of the original ones, maintain these properties for simple reasons.) The proof at this point is the same as that in Corollary \ref{cor:pbg}: algebraic fibers of $\Pi(m+1)$ of type $F_n$ but not $FP_{n+1}$ for $0 \leq n \leq m-1$ arise as kernels of discrete characters pulled back from $\Pi(m)$, while the algebraic fibers of type $F_{m}$ arise as kernels of a discrete character of the type determined by Theorem  \ref{thm:poly1}. As the Euler characteristic of $\Pi(m+1)$ is nonzero, this algebraic fiber is not $FP_{m+1}$.
\end{proof}

\begin{remark} One may venture  
 and ask whether the only normal $FP_m$-subgroups in poly-free or polysurface groups of length $m$ with nonzero Euler characteristic must have finite index in one element of a subnormal series of some finite-index subgroup. This would be, in a sense, a generalization of the aforementioned f.g.n. property for surface group, and it is known to hold in the case of length $2$ both for poly-free groups (see \cite[Corollaries 8.6 and 8.7]{Bieribook}) and, under the slightly stronger $FP_3$-assumption, for  poly-surface groups (see \cite[Theorem 3.10]{Hibook}).  
\end{remark} 

We want to finish this section on algebraic fibrations of poly-free and poly-surface groups with some comparison with those of (virtually) RFRS groups. This class of groups, which contains the direct products of free and of surface groups, is prone to admit, at least virtually, algebraic fibrations. More precisely, we have the following result of Fisher, that generalizes previous groundbreaking work of Kielak (see \cite{Fisher,Kielak}).
\begin{theorem} (Kielak, Fisher) \label{thm:kf} Let $G$ be a virtually RFRS group of type $FP_n(\Q)$. Then there exists a finite index subgroup of $G$ that admits an algebraic fibration of kernel $FP_n(\Q)$ if and only if the $L^2$-Betti numbers $b_{p}^{(2)}(G)$ vanish for $p = 0,...,n$.  \end{theorem}

We can observe some similarity of this result with those we obtained for poly-free and poly-surface groups (which are $FP$, hence {\it a fortiori} $FP_{n}(\Q)$) especially in light of the following result, that is certainly well-known (see e.g. \cite{Ga}):

\begin{proposition} \label{prop:l2}Let $G$ be a poly-free or a poly-surface group of length $m$ with subnormal filtration $1  = G_0 \unlhd G_1 \unlhd \ldots \unlhd G_m = G$; then for $0 \leq p \leq m-1$ we have $b_{p}^{(2)}(G) = 0$ while $b_{m}^{(2)}(G) = (-1)^{m}\chi(G)$. 
\end{proposition}

\begin{proof} We proceed by induction on $1 \leq i \leq m$. The statement is well-known to hold for $G_1$. For each $2 \leq i \leq m$, $G_i$ is an extension of the infinite group $G_{i-1}$ with quotient $G_{i}/G_{i-1}$ that has an element of infinite order. By the inductive hypothesis  we have $b_{p}^{(2)}(G_{i-1}) = 0$ for $0 \leq p \leq i-2$ while $b_{i-1}^{(2)}(G_{i-1}) = (-1)^{i-1}\chi(G_{i-1}) < \infty$, hence we can apply \cite[Theorem 7.2(6)]{Luck} to deduce that for $0 \leq p \leq i-1$ we have $b_{p}^{(2)}(G_{i}) = 0$. For dimensional reasons, supplemented by Poincar\'e duality in the poly-surface case, the only nontrivial $L^2$-Betti number of $G_i$ can therefore only be $b_{i}^{(2)}(G_{i})$. The $L^2$-Euler characteristic of $G_i$, which coincides with $\chi(G_i)$, satisfies the relation \[ \chi(G_i) = \sum_{p} (-1)^p b_{p}^{(2)}(G_{i}) = (-1)^i b_{i}^{(2)}(G_{i}). \] \end{proof}
\begin{remark} There is a little discrepancy between the result of Theorem \ref{thm:poly} and Theorem \ref{thm:kf}, inasmuch as the result of Theorem \ref{thm:poly} yields a slightly stronger homological finiteness (over $\Z$ instead than over $\Q$), but also a slightly weaker obstruction (again over $\Z$ instead than over $\Q$). The latter problem can be quite easily addressed, bringing the obstruction at par by making use of  $L^2$-techniques applied to algebraic fibrations as in \cite[Theorem 7.2(5)]{Luck} 
\end{remark}

In light of Proposition \ref{prop:l2} algebraic fibrations of groups like those covered by Theorem \ref{thm:poly} and Corollaries \ref{cor:pbg} and \ref{cor:itko} seem to have a behavior similar to that of virtually RFRS groups. However, expecting that behavior to hold for all poly-free and poly-surface groups would be quite incorrect, although examples of that are not easy to build. We will at least in part do so, piggybacking on a examples of free-by-free groups with no virtual excessive homology. In \cite{K-V-W} there are examples of semidirect products of the form $K_{m,n}= F_{m} \rtimes F_n$, for any $m \geq 5$ and $n \geq 2$, that have no virtual excessive homology. Their first  BNS invariant $\Sigma^{1}(K_{m,n})$ is empty, as it coincides with that of $F_n$. Direct products of the form $K_{m,n} \times F_{s}$, $s \geq 2$ (which are poly-free groups of length $3$) have the second BNSR invariant $\Sigma^2(K_{m,n} \times F_{s};\Q)$ that can be computed using the product formula (\cite[Theorem 1.3]{B-G},: precisely, we have
\[ S(K_{m,n} \times F_{s}) \setminus \Sigma^2(K_{m,n} \times F_{s};\Q) = \bigcup_{0 \leq p \leq 2} (S(K_{m,n})  \setminus \Sigma^{p}( K_{m,n};Q)) \ast (S(F_{s})   \setminus \Sigma^{2-p}(F_{s};\Q)) =   S( K_{m,n} \times F_{s})  \] hence $\Sigma^2(K_{m,n} \times F_{s};\Q) = \emptyset$, which entails that any algebraic fibrations of $K_{m,n} \times F_{s}$ (whose existence is guaranteed by the product formula for 
$\Sigma^1(K_{m,n} \times F_{s};\Z)$ or, if preferred, by Theorem \ref{Main1}) has kernel that is not $FP_2(\Q)$. This property continues to hold even virtually. In fact, any finite index subgroup of $K_{m,n} \times F_{s}$ admit a finite index subgroup which is direct product of a finite index subgroup of $K_{m,n}$ (and therefore has no virtual excessive homology) and  a finite index subgroup of $F_s$ so that even for these subgroups $\Sigma^{2}$ is empty. Algebraic fibrations pass to finite index subgroups preserving finiteness properties of the kernel, so any virtual algebraic fibration of $K_{m,n} \times F_{s}$ has kernel that is not $FP_2$,  which runs against the behavior of virtually RFRS groups. This entails also that $K_{m,n} \times F_{s}$ is incoherent, with the incoherence witnessed by an algebraic fiber. The fact that $K_{m,n}$ (and any direct product containing it) is incoherent is already known from \cite[Lemma 4.2]{K-W} as its monodromy is not injective, hence it contains a $F_2 \times F_2$ subgroup. This entails that $K_{m,n} \times F_{s}$ fails to be $2$-coherent or homologically $2$-coherent as well.

Examples of this type with poly-surface groups seem harder to obtain: in fact it is still an open problem whether there exist surface-by-surface groups with no virtual excessive homology, see \cite{K-V-W}. 
In principle, if one were capable to gain information of the finiteness properties of (virtual) algebraic fibrations of poly-free and poly-surface groups, one may use Theorem \ref{thm:kf} as an obstruction to virtual RFRSness. In practice, this clashes with the fact that (with the exception of the product case just considered) it seems exceedingly hard to gain that information.

\end{document}